\documentclass[10pt]{amsart}
\usepackage{amsmath}
  \usepackage{graphics} 
\usepackage{graphicx}  
  \usepackage{epsfig} 
 \usepackage[colorlinks=true]{hyperref}
  \usepackage{float}
  \usepackage{caption}
\usepackage[labelsep=quad,indention=10pt]{caption}
\usepackage[labelfont=bf,list=true]{subcaption}
\usepackage{amsmath}
\usepackage{enumitem}  
\usepackage{wrapfig}
\usepackage{graphicx}
\usepackage{subcaption}
\usepackage{lscape}
\usepackage{amssymb}

\usepackage{amsthm}
\theoremstyle{plain}

\newtheorem*{theorem*}{Theorem}
\newtheorem{theorem}{Theorem}[section]

\newtheorem{lemma}[theorem]{Lemma}

\theoremstyle{definition}
\newtheorem{definition}[theorem]{Definition}
\newtheorem{remark}{Remark}
\newtheorem*{proof*}{Proof}

  \title[ additional food ] 
    {Novel dynamics in an additional food provided predator-prey system with mutual interference}

\begin{document}
\maketitle

\centerline{\scshape    Sureni Wickramsooriya$^{1}$, Jonathan Martin$^{2}$, Aniket Banerjee$^{3}$ and Rana D. Parshad$^{3}$}
\medskip
{\footnotesize

   \medskip

 \centerline{1) Department of Pathology, Microbiology and Immunology,}
 \centerline{University of California-Davis,}
   \centerline{ One Shields Avenue, Davis, CA 95616, USA.}

   \medskip
  
 \centerline{2) Department of Mathematics, }
 \centerline{Clarkson University, }
   \centerline{ 8 Clarkson Ave, Potsdam, NY 13699, USA.}  
   \medskip
     \centerline{ 3) Department of Mathematics,}
 \centerline{Iowa State University,}
   \centerline{Ames, IA 50011, USA.}
      \medskip
 }

\begin{abstract}


The provision of additional food (AF) sources to an introduced predator has been identified as a mechanism to improve pest control.
However, AF models with prey dependent functional responses can cause unbounded growth of the predator \cite{S27}. To avoid such dynamics, an AF model with mutual interference effect has been proposed \cite{S02}. The analysis therein reveals that if the quantity of additional food $\xi > h(\epsilon)$, where $\epsilon$ is the mutual interference parameter, then pest eradication is possible, and this is facilitated via a transcritical bifurcation.
We revisit this model and show novel dynamical behaviors. In particular, pest eradication is possible for a tighter range of AF  $g(\epsilon) < \xi < f(\epsilon) < h(\epsilon)$, and can also occur via a saddle node bifurcation.
We observe bi-stability, as well as local bifurcations of Hopf  type. We also prove a global bifurcation, of homoclinic type. This bifurcation in turn is shown to create a non-standard dynamic wherein the pest extinction state becomes an ``almost" global attractor. 
To the best of our knowledge, this is the first proof of existence of such a dynamical structure in AF models.
We discuss our analysis in the context of designing novel bio-control strategies.

\end{abstract}

 \section{Introduction}

Biotic stressors such as pests, pathogens, and weeds have become a threat to the crop yield and efficacy of food products in various ways. For instance, unwanted harmful pests  cause considerable losses of crop yield, both in quantity and quality. The main source to overcome this issue has been chemical measures; particularly including pesticides and insecticides. The consumption of synthetic pesticides such as DDT, BHC, aldrin, dieldrin, captan, and 2,4-D has accelerated by the 1940s (\cite{S03}). However, these pesticides/insecticides have many harmful consequences and negative drawbacks both from an economical and an environmental point of view. Particularly, they contaminate the environment, remaining there for an indefinitely long period. Also, many of them are hazardous chemicals that can cause incredibly high risk even to human health. One potential alternative strategy to this is releasing a natural enemy of the pest to control them. From the modeling point of view, this is simply a predator-prey relationship where, the predator population can grow, reproduce and perhaps eliminate the pest from the system via depredation. They are primarily three ways to introduce the natural enemy concept in the field; conservation of existing natural enemies, introducing new natural enemies and establishing a permanent population (classical biological control), mass rearing, and periodic release, on a seasonal basis (\cite{S04},\cite{S05}, \cite{S06}). A possible drawback of this approach tends to be incomplete elimination of pests due to an insufficient depredation rate. One way to boost the depredation rate is by providing them with an alternative food source (\cite{S07}, \cite{S08}, \cite{S09}, \cite{S10}, \cite{S11}). Some studies such as \cite{S04}, \cite{S05}, \cite{S06} point out that providing non-prey food items would  magnify the abundance, and life history traits, such as longevity, fecundity, or foraging behavior of predators; typically which would directly contribute to the efficacy of depredation.\\

The provision of AF may or may not successfully control the target pest. The quality and quantity of alternative food sources play a vital role in the outcomes of biological pest control targets. In this context, effort has been put into exploring several strategies for integrated pest management programs. To this end mathematical models that incorporate AF for the predator in predator-prey systems with have been intensely studied (\cite{S12}, \cite{S13}, \cite{S42}, \cite{S16}, \cite{S41}). Although these show that  pest extinction can be achieved \cite{S12}, there are some possible complications with such an approach. One potential unrealistic feature of these models is the unbounded growth of the introduced predator population, in certain parameter regimes. This would negatively impact sustainability of the ecosystem via non-target effects \cite{S47}. In addition, \cite{S19} has pointed out that the provision of high quality supplementary food source can cause apparent competition between the two food sources. Consequently, a predator might reduce target prey consumption; perhaps even switch completely to the alternative food source. \\

To avoid the outcome of unbounded predator growth, a possible solution is introducing a limiting factor to this growth (\cite{S16}). The limiting behavior can be modeled into the predator dynamics via mutual interference. Mutual interference \cite{S32a} is defined as the behavioral interactions among feeding organisms, that reduce the time that each individual spends obtaining food, or the amount of food each individual consumes. It occurs most commonly where the amount of food is scarce, or when the population of feeding organisms is large. Early work on mutual/predator interference \cite{S48}, modeled the interference term as $\left(\frac{x}{1+hx}\right)y^{m}$, where $y$ is the predator density, $x$ is prey density, $h$ is the handling time, or the time it takes the predator to kill and ingest the prey, and $0 < m < 1$ is the interference parameter. Others have modeled interference as $\left(\frac{x}{1+hx}\right)^{m_{1}}$ \cite{S55}. One could consider this situation to model a predator with a greater feeding rate, or a more aggressive predator, than one in which $m_{1} = 1$, the classical case. This is clear from simple comparison, $p(x)|_{m_{1}=1} < p(x)|_{0 < m_{1} < 1}$, $\forall x > 0$. Other forms of interference consider the ratio dependent response, $p(x) = \left(\frac{x}{y+x}\right)$, and it is argued that this response is closer to the true predator-prey dynamics of natural systems, mechanistically, empirically as well as behaviorally. Another form of interference is modeled via the Beddington-Deangelis functional response \cite{S33}. The systematic study by Prasad et al. (\cite{S02}) extensively focuses on incorporating a predator mutual interference effect into Srinivasu's original model \cite{S12}. In this modified system, the Holling type II functional response is replaced to incorporate mutual interference by adopting precisely a Beddington-Deangelis functional response. Prasad et al. investigate many interesting dynamical phenomena in this model, including stable coexistence at a lower abundance of prey, by varying the mutual interference term. Their analysis is focused on one coexistence state.


\section{Prior Results and Motivation}

\subsection{A recap of pest dependent responses}
The following \emph{general} model for an introduced predator population $y(t)$ depredating on a target pest population $x(t)$, also provided with an additional food source, has been proposed in the literature \cite{S12, S40}. 
\begin{equation}
\label{Eqn:1g}
\frac{dx}{dt} = F(x,y) =  x\left(1-\frac{x}{\gamma}\right) - f(x,\xi,\alpha) y, \    \frac{dy}{dt} = G(x,y) =  g(x, \xi, \alpha, \beta) y  - \delta y.
\end{equation}
Herein the parameters $\xi, \alpha, \beta, \gamma, \delta$ are all positive constants. $\gamma$ is the carrying capacity of the pest population. $\beta$ is the conversion efficiency of the predator. $\delta$ is the death rate of the predator. $\frac{1}{\alpha}$ is the quality of the additional food provided to the predator and $\xi$ is the quantity of additional food provided to the predator. 
Here $f(x,\xi,\alpha)$ is the functional response of the predator, that is pest dependent but also dependent on the quality and quantity of additional food. Likewise, $g(x,\xi,\alpha)$ is the numerical response of the predator. 
If $\xi = 0$, that is there is no additional food, the model reduces to a classical predator-prey model of Gause type, and $f(x,0,0) = f(x) = g(x,0,0)$. Herein $f$ has the properties of a standard pest/prey dependent functional response.
\begin{remark}
When $\xi=0$ the pest extinction state $(0,y^{*})$ does not exist, and the complete extinction state $(0,0)$, is typically unstable \cite{S12}. Thus engineering 
$f(x,\xi,\alpha), g(x,\xi,\alpha)$ as a means to achieve a pest free state, has immense practical value, and has been well studied. Table \ref{Table:3},  summarizes some of the key results in the literature, in terms of the functional forms used in these models.
\end{remark}
\begin{table}[H]
\caption{Dynamics of AF models}\label{Table:3}
	\scalebox{0.66}{
{\begin{tabular}{|c|c|c|c|c|}
		\hline
		& \mbox{Functional form} & \mbox{Relevant literature} & \mbox{AF requirement}& \mbox{ Effect on pest control}\\
		& &  & &  \\ \hline \hline
		(i) & $f(x,\xi,\alpha) = \frac{x}{1+\alpha \xi + x}$ & \cite{ S40} & $\xi > \frac{\delta}{\beta - \delta \alpha}$ & \mbox{Pest is eradicated, switching AF}\\
		& $g(x,\xi,\alpha) = \frac{\beta (x+\xi)}{1+\alpha \xi + x}$ &  &  &  \mbox{maintains/eliminates predator }\\          \hline
		(ii) & $f(x,\xi,\alpha) = \frac{x^{2}}{1+\alpha \xi^{2} + x^{2}}$  & \cite{S41}  & $\xi >\sqrt{ \frac{\delta}{\beta - \delta \alpha}}$  &  \mbox{Pest is eradicated, switching AF}\\
		& $g(x,\xi,\alpha) = \frac{\beta (x^{2}+\xi^{2})}{1+\alpha \xi^{2} + x^{2}}$ &  & & \mbox{ maintains/eliminates predator } \\ \hline
		(iii) & $f(x,\xi,\alpha) = \frac{x}{(1+\alpha \xi)( \omega x^{2}+1) + x}$   & \cite{S42} & $\xi > \frac{\delta}{\beta - \delta \alpha}$ &  \mbox{Pest is eradicated, switching AF} \\
		& $g(x,\xi,\alpha) = \frac{\beta (x+\xi(\omega x^{2}+1)}{(1+\alpha \xi)(\omega x^{2} + 1) + x}$ &  &  & \mbox{ maintains/eliminates predator}  \\ \hline
		(iv) & $f(x,\xi,\alpha) = \frac{x}{1+\alpha \xi + x + \epsilon y}$ & \cite{S02} &  $\beta \xi = \delta(1+\alpha)$  & \mbox{Pest extinction state stabilises,}  \\
		& $g(x,\xi,\alpha) = \frac{\beta (x+\xi)}{1+\alpha \xi + x + \epsilon y}$ &  & $\beta(\gamma + \xi) =\delta(1+\alpha \xi + \gamma)$  &\mbox{via a saddle-node bifurcation}  \\ \hline
		(v) & $f(x,\xi,\alpha)$ as in (i) + Allee effect in pest & \cite{S43} &  pest extinction state impossible & \mbox{Pest extinction is hindered}  \\
		& $g(x,\xi,\alpha)$ as in (i)  &    &\mbox{via Allee effect}  \\ \hline
		\end{tabular}}}
\vspace{.1cm}		 
     \flushleft
     \end{table}

		 The idea in the works in Table \ref{Table:3}, is to implement a pest management strategy, where  the quantity of AF ($\xi$) is increased, so that $\xi > \xi_{critical}$, where $\xi_{critical}$ depends on the model parameters. Note, $\xi = \xi_{critical}$, is essentially the equation of the predator ($y$-axis, $x=0$) axis. Thus, choosing $\xi > \xi_{critical}$, pushes the vertical predator nullcline to the left past the predator (y-)axis, into the $2^{nd}$ quadrant - so there is no positive interior equilibrium. Via positivity of solutions, trajectories starting in the positive ($1^{st}$) quadrant will move towards the predator axis and ``hit" it, yielding pest extinction.

When the functional response $f$ is of type II, $\xi_{critical} =  \frac{\delta}{\beta- \delta \alpha}$. Thus, the pest cannot be driven extinct for any $ \xi \in  [0, \frac{\delta}{\beta - \delta \alpha})$.
The literature \cite{S07, S10, S11} states that increasing AF beyond this interval, so if $\xi \in ( \frac{\delta}{\beta- \delta \alpha}, \infty )$,
eradicates the pest from the ecosystem in a finite time, and the predators survive only on the AF subsequently. It should be noted that these works \cite{S07, S10, S11} also caution against the possible unbounded growth of the predator. However, this is not proved rigorously. The finite time extinction is not an accurate representation of the dynamics. This was shown in \cite{S44}. Recently we also show that providing AF in the large regime ($\xi \in (\frac{\delta}{\beta- \delta \alpha}, \infty )$) can drive the pest extinct asymptotically (c.f. [Proposition 1, \cite{S44}]) - but will always result in blow-up in infinite time of the predator density \cite{S27}. The results of \cite{S27} are applicable to  the type III response as well.

\subsection{The mutual interference model}
In order to prevent the unbounded predator growth several alternative mechanisms have been proposed.

\begin{itemize}
\item[(i)] The inhibitory effect produced by prey defense, such as via the type IV functional response \cite{S41, S49}.
\item[(ii)] The inhibitory effects produced by predator interference, such as via the Beddington-Deangelis functional response \cite{S02}.
\item[(iii)] The inhibitory effect produced by a purely ratio dependent response \cite{S45, S46}.
\item[(iv)] The inhibitory effect produced by intraspecific predator competition \cite{S27}.
\end{itemize}

All of the above prevent unbounded growth of the predator.
We focus on the AF model that has been proposed by Prasad et al \cite{S02}, described via the following system of differential equations \eqref{model1}. 

\begin{equation}\label{model1}
 \frac{dx}{dt} = x\left(1-\frac{x}{k}\right) - \frac{xy}{1+\alpha \xi+x+\epsilon y}, \
       \frac{dy}{dt} =\frac{\beta (x+\xi)y}{1+\alpha \xi+x+\epsilon y}-\delta y
 \end{equation}

This model was essentially developed by combining the additional food model that was originally derived by Srinivasu et al. \cite{S12}, \cite{S40} with the Beddington–DeAngelis model \cite{S32}, \cite{S33}, \cite{S34} to replace the Holling type II functional response. The target prey density $x$ interacts with the predator whose total population density is denoted by $y.$ The predator is provided with an alternative non prey food source whose quality inversely proportional to the parameter $\alpha$ and the quantity is measured by $\xi.$ The parameter $\gamma$ represents the carrying capacity of the prey and $\beta$ and $\delta$ represent the birth and death rates of predator respectively. The Beddington–DeAngelis functional response incorporates with the mutual interference by the term $\epsilon y$ where the parameter $\epsilon $ represents mutual interference effect among the predators.

Since the system ~\eqref{model1} is Kolmogorov type, the model includes the positively invariant set defined by,
$\phi=\{(x,y)\in {\rm I\!R}\times{\rm I\!R} |  x\geq 0, y\geq 0\}.$ 
\subsection{Past results on mutual interference model}

The analysis in \cite{S02} focused on two separate cases, (i) high interference $\epsilon > 1$ (ii) low interference $0<\epsilon <1$. In (i) there is always   one unique interior equilibrium - if a feasible pest free equilibrium exists it is a saddle. Thus from the point of view of pest eradication, high predator interference is not desirable. In that no matter what quality and/or quantity of additional food are chosen, the pest eradication state cannot be (even locally) attracting. In case (ii) it was shown therein that,

\begin{itemize}

\item[(i)] The predator density is always bounded. It cannot grow past a uniform (in parameter) bound, no matter how the parameters are changed, so as to engineer control of the pest.

\item[(ii)] There can exist one unique interior equilibrium, as long as $\xi<\frac{\delta}{\beta-\delta\alpha-\beta\epsilon}$. 

\item[(iii)] A pest free equilibrium can coexist with the interior equilibrium - in that case the pest free equilibrium is always a saddle.

\item[(iv)] The pest free state is globally stable if $\xi>\frac{\delta}{\beta-\delta\alpha-\beta\epsilon}$. This change of stability occurs through a transcritical bifurcation.

\item[(v)] There can also exist a predator free equilibrium.
\end{itemize}

\subsection{Current approach}

The analysis in \cite{S02} is based on the geometric assumption that the slope of the predator nullcline (say $m_{1}$) in \eqref{model1} is greater than the slope of the tangent line to the prey nullcline (say $m$) at 
$(0,\frac{1+\alpha \xi}{1-\epsilon})$,  that is where it cuts the predator (y) axis. After a few computations this is expressed as,

\begin{equation} \label{eqslopecom1}
    m  <  m_{1} \Longleftrightarrow  
\frac{k(1-\epsilon)-(1+\alpha\xi)}{k\left(1-\epsilon\right)^2} < \frac{\beta-\delta}{\delta \epsilon},
\end{equation}
this yields,
\begin{equation} \label{eqslopecom2}
 \xi>\frac{k(1-\epsilon)\left[\delta-\beta(1-\epsilon)\right]-\delta \epsilon}{\alpha \delta \epsilon}.
\end{equation}

This is the reason that the authors in \cite{S02} derive global stability of the pest free state, when
the condition in (iv) above holds, see Fig. \ref{figure3_2}. This condition follows by choosing the $y$-intercept of the predator nullcline to be large than where the prey nullcline intersects the $y$-axis. This in conjunction with $m_{1} > m$, yields global stability of the pest free state.

This is best summarised in Fig. \ref{figuretwoInt2}. Thus, the fundamental premise of \cite{S02}, is to assume AF introductions $\xi$, such that $\xi > h(\epsilon)$, see Fig. \ref{figuretwoInt2}. In the event that this assumption is changed, novel dynamics (not reported in \cite{S02}) are possible - which yield various novel bio-control scenarios of interest.

The primary aim of this manuscript is as follows,
\begin{itemize}

\item In the current work we consider $\xi < h(\epsilon)$, and provide a much tighter window on the quantities of AF needed for pest eradication. In particular, we are able to derive pest extinction for $h(\epsilon) > f(\epsilon) > \xi > g(\epsilon)$, see Fig. \ref{figuretwoInt2}.

\item To this end we perform a complete analysis of the model proposed in \cite{S16} to recover all the dynamic behaviors of this system. 

\item In particular, we find that the system can possess up to two coexistence equilibrium states. Importantly, the system experiences bi-stable behavior when the system has two coexistence equilibrium states. We discuss several implications of this for bio-control.


\item We explored the dynamical behaviors of these two coexistence states, which revealed standard local bifurcation behaviour such as Hopf bifurcation and saddle node bifurcation. We show that the change of stability of the axial (pest-free) equilibrium from local to global can occur through the saddle node bifurcation - which is a different route to the stability change than the transcritical bifurcation route reported in \cite{S16}.

\item Non standard and global bifurcation behavior such as Homoclinic bifurcation is also shown. This particular bifurcation has several novel implications for bio-control - in particular it leads to the pest free state becoming an ``almost" global attractor.

\item The local stability of the interior equilibrium state does \emph{not} imply its global stability. This is seen from the bi-stability behavior in the case of two coexistence equilibrium states. It is conjectured to be true even in the case of one coexistence equilibrium.

\end{itemize}

\begin{figure}\label{figuretwoInt2}
\centering
\includegraphics[width=12cm]{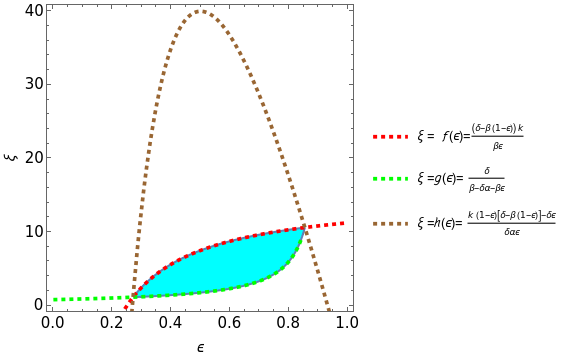}
\caption{The coefficients $b$ and $c$ of the quadratic equation \eqref{eqn:chap3_22} are expressed through the equations, $\xi=\frac{\left[\delta-\beta(1-\epsilon)\right]k}{\beta\epsilon}$ (dashed red) and $\xi=\frac{\delta}{\beta-\delta\alpha-\beta\epsilon}$ (dashed green) in the $(\epsilon,\xi)$ parametric plane. The parametric constraint for which the slope of predator nullcline and  the slope of prey nullcline at the point $\left(0,\frac{1+\alpha\xi}{1-\epsilon} \right)$  are  equivalent $(m=m_1)$ is illustrated by the dashed brown line. The common intersection of all three curves where the parametric space is $\frac{\delta}{\beta-\delta\alpha-\beta\epsilon}<\xi< \frac{\left[\delta-\beta(1-\epsilon)\right]k}{\beta\epsilon},$ is highlighted in light blue. This region corresponds to the system having two positive interior equilibrium points.}
\label{figuretwoInt2}
\end{figure}


We discuss details in the subsequent section.
\section{Further analysis of the mutual interference model}
\subsection{Existence of equilibrium points}\label{model2EqAna}
In this section, the equilibrium analysis of the  model ~\eqref{model1} is performed with an emphasis on the existence of 
two positive interior equilibrium points. We consider the biologically feasible parametric constraints, $\beta> \delta$ and $k>0$ throughout the analysis.

 
The equilibrium solutions of a system (\ref{model1}) are obtained by solving
$\frac{dx}{dt}=\frac{dy}{dt}=0$.\\

\begin{equation}\label{eqn:chap3_17}
\frac{dy}{dt} =\frac{\beta (x+\xi)y}{1+\alpha \xi+x+\epsilon y}-\delta y=\frac{\beta (x+\xi)y-\delta y\left(1+\alpha \xi+x+\epsilon y \right) }{1+\alpha \xi+x+\epsilon y}=0
\end{equation}

Then we yield the equation (\ref{eqn:chap3_18}),
\begin{equation}\label{eqn:chap3_18}
y \left[ \beta (x+\xi)-\delta \left(1+\alpha \xi+x+\epsilon y \right)\right]=0
\end{equation}
The non trivial solutions of the equation (\ref{eqn:chap3_18}) are as follows,
\begin{equation}\label{eqn:chap3_19}
  y=\frac{\left(\beta-\delta\right)}{\delta\epsilon} x+ \frac{\left(\beta-\delta\alpha\right)\xi-\delta}{\delta\epsilon}
\end{equation}
\begin{equation}\label{eqn:chap3_20}
\frac{dx}{dt} = x\left(1-\frac{x}{k}\right) - \frac{xy}{1+\alpha \xi+x+\epsilon y}=
\frac{x\left(k-x\right)\left( 1+\alpha \xi+x+\epsilon y  \right) -k xy}{k\left( 1+\alpha \xi+x+\epsilon y  \right)}=0
\end{equation}

\begin{equation}\label{eqn:chap3_21}
 \left(k-x\right)\left( 1+\alpha \xi+x+\epsilon y  \right) -k y = 0
\end{equation}
 Now, we substitute $y=\frac{\left(\beta-\delta\right)}{\delta\epsilon} x+ \frac{\left(\beta-\delta\alpha\right)\xi-\delta}{\delta\epsilon}
$ into the equation (\ref{eqn:chap3_21}) in order to derive the following quadratic equation of $x$.
 \begin{equation}\label{eqn:chap3_22}
 \beta\epsilon x^{2}+\left[\beta\epsilon\xi-\left[\delta-\beta(1-\epsilon)\right]k\right]x+\left[\left(\beta-\delta\alpha-\beta\epsilon\right)\xi-\delta\right]k= 0
\end{equation}

The analysis presents  multiple equilibrium points including trivial, axial and coexistence states that are listed below as follows:
\begin{itemize}
    \item Trivial equilibrium,$E_{(0,0)}=\left(0,0\right)$
    
     
    \item Predator free equilibrium, $E_{PredEx}=(k,0)$
    
    
    \item Prey free equilibrium, $E_{PreyEx}=\left(0,\frac{\left(\beta-\alpha\delta\right)\xi-\delta}{\delta\epsilon}\right)$

When the parameter corresponding to additional food quantity $(\xi)$ is increased, the prey free equilibrium lifts up along the predator axis ($y-$axis). As long as the amount of additional food density $(\xi)$ exceeds $\frac{\delta}{\beta-\delta\alpha},$ the prey free equilibrium point is positive. 
\item  The quadratic equation(\ref{eqn:chap3_22}) can be precisely formulated as a standard form of $ax^2+bx+c=0$
where, $a=\beta\epsilon,\hspace{3mm}
b=\beta\epsilon\xi-\left[\delta-\beta(1-\epsilon)\right]k,\hspace{3mm}
c=\left[\left(\beta-\delta\alpha-\beta\epsilon\right)\xi-\delta\right]k $. The solutions of (\ref{eqn:chap3_22}) defined the coexistence equilibrium states $E_1=(x_1,y_1)$ and $E_2=(x_2,y_2)$ with the components, $x_1^{*}=\frac{-b-\sqrt{\Delta}}{2a}=\frac{-\left[\beta\epsilon\xi-\left[\delta-\beta(1-\epsilon)\right]k\right]-\sqrt{\Delta}}{2\beta\epsilon}$ and 
$x_2^{*}=\frac{-b+\sqrt{\Delta}}{2a}=\frac{-\left[\beta\epsilon\xi-\left[\delta-\beta(1-\epsilon)\right]k\right]+\sqrt{\Delta}}{2\beta\epsilon}$ 
where, $y_i^{*}=\frac{\left(\beta-\delta\right)}{\delta\epsilon}x_i^{*}+ \frac{\left(\beta-\delta\alpha\right)\xi-\delta}{\delta\epsilon}$ with $i=1,2.$ Here, the discriminant of the quadratic equation (\ref{eqn:chap3_22}) is, 
\begin{equation}
    \Delta= b^2-4ac=\left[\beta\epsilon\xi-\left[\delta-\beta(1-\epsilon)\right]k\right]^2+4\beta\epsilon k\left[\left(\beta-\delta\alpha-\beta\epsilon\right)\xi-\delta\right]
\end{equation}  
\end{itemize}

The primary objective of our work is to investigate the existence of two positive coexistence states. To achieve this, we determine the parametric restrictions required for the existence of two positive interior equilibrium points using Descartes' rule of signs.
     \begin{itemize}
      
    \item $\Delta>0  \hspace{4mm}\Rightarrow$ \\
    $\left[\beta\epsilon\xi-\left[\delta-\beta(1-\epsilon)\right]k\right]^2+4\beta\epsilon k\left[\left(\beta-\delta\alpha-\beta\epsilon\right)\xi-\delta\right]>0$\\
   \vspace{1mm}     
    \item  $\frac{-b}{a}>0$ and $\frac{c}{a}>0$ that is, $b<0 \hspace{4mm}$ and $c>0
    \Rightarrow $

    \begin{equation}
         \frac{\delta}{\beta-\delta\alpha-\beta\epsilon}<\xi<\frac{\left[\delta-\beta(1-\epsilon)\right]k}{\beta\epsilon}
    \end{equation} 
   
\vspace{1mm}
   \item  $y_i^{*}>0 \hspace{4mm}\Rightarrow \hspace{4mm} x_i^{*}>\frac{\delta-(\beta-\delta\alpha)\xi}{k(\beta-\delta)}$
      
\end{itemize}

 According to the above classification, the system essentially has two interior equilibrium points for the parameter space $\frac{\delta}{\beta-\delta\alpha-\beta\epsilon}<\xi<\frac{\left[\delta-\beta(1-\epsilon)\right]k}{\beta\epsilon}.$ On the other hand, there exists only one interior equilibrium point as long as $\xi<\frac{\delta}{\beta-\delta\alpha-\beta\epsilon}$(see Figure \ref{figuretwoInt2}.)



\subsection{Local Stability Analysis} \label{model2StAna}
In this section, we explore the parametric constraints under which the interior equilibrium points are locally asymptotically stable or unstable. This is achieved through linear stability analysis, which involves evaluating the trace and determinant of the Jacobian Matrix at each interior equilibrium point.The standard Jacobian Matrix $\left(J\right)$ at the point $\left(x_2^{*},y_2^{*}\right)$ is evaluated as follows.\\
\begin{equation}
   J =\begin{bmatrix}
1-\frac{2x_2^{*}}{k}-\frac{\left(1+\alpha\xi+\epsilon y_2^{*}\right)y*}{\left(1+\alpha\xi+x_2^{*}+\epsilon y_2^{*}\right)^2}&
  -\frac{\left(1+\alpha\xi+ x_2^{*}\right)x_2^{*}}{\left(1+\alpha\xi+x_2^{*}+\epsilon y_2^{*}\right)^2} \\\\
\frac{\left(1+\alpha\xi+ \epsilon y_2^{*}-\xi\right)\beta y_2^{*}}{\left(1+\alpha\xi+x_2^{*}+\epsilon y_2^{*}\right)^2} & 
\frac{\beta\left(x_2^{*}+\xi \right) \left(1+\alpha\xi+x_2^{*}\right)}{\left(1+\alpha\xi+x_2^{*}+\epsilon y_2^{*}\right)^2}-\delta
\end{bmatrix}=\begin{bmatrix} J_{11} &J_{12}\\J_{21}&J_{22} \end{bmatrix}
\end{equation}


We conduct the linear stability analysis for the interior equilibrium points $E_1=\left(x_1^{*},y_1^{*}\right)$  and  $E_2=\left(x_2^{*},y_2^{*}\right)$. As the first step we derive the trace and determinant at the point $\left(x_i^{*},y_i^{*}\right)$ for $i=1,2.$

 \begin{equation}\label{trace}
Tr_ {(x_i^{*},y_i^{*})}=
 1+\frac{\beta\left(x_i^{*}+\xi\right)\left(1+\alpha\xi+x_i^{*}\right) }{\left(1+\alpha\xi+x_i^{*}+\epsilon y_i^{*}\right)^2}-\left[\frac{2x_i^{*}}{k}+\frac{\left(1+\alpha\xi+\epsilon y_i^{*}\right)y_i^{*}}{\left(1+\alpha\xi+x_i^{*}+\epsilon y_i^{*}\right)^2}+\delta\right]
 \end{equation}

\begin{equation}\label{det}
\begin{split}
    det_ {(x_i^{*},y_i^{*})} & = \left[1-\frac{2x_i^{*}}{k}-\frac{\left(1+\alpha\xi+\epsilon y_i^{*}\right)y_i^{*}}{\left(1+\alpha\xi+x_i^{*}+\epsilon y_i^{*}\right)^2}\right].\left[\frac{\beta\left(x_i^{*}+\xi\right)\left(1+\alpha\xi+x_i^{*}\right) }{\left(1+\alpha\xi+x_i^{*}+\epsilon y_i^{*}\right)^2}-\delta\right]+\\ &
    \left[\frac{\left(1+\alpha\xi+ \epsilon y_i^{*}-\xi\right)\beta y_i^{*}}{\left(1+\alpha\xi+x_i^{*}+\epsilon y_i^{*}\right)^2}\right].\left[\frac{\left(1+\alpha\xi+ \epsilon y_i^{*}-\xi\right)\beta y_i^{*}}{\left(1+\alpha\xi+x_i^{*}+\epsilon y_i^{*}\right)^2}\right]
    \end{split}
\end{equation}

Since the point $(x_i^*,y_i^*)$ is on the non trivial predator nullcline, the points satisfy the equation \eqref{eqn:chap3_19}, thus we yield the following simplified expressions.

\begin{equation}\label{proof3.1_3}
    1+x_i^{*}+\alpha\xi+\epsilon y_i^{*}=1+x_i^{*}+ \alpha\xi+\epsilon\frac{\left(\beta-\delta\right)x_i^{*}}{\delta\epsilon}+\epsilon\frac{\left(\beta-\delta\alpha\right)\xi-\delta}{\delta\epsilon}=
\frac{\beta\left(x_i^{*}+\xi\right)}{\delta}
\end{equation}

 \begin{equation}\label{proof3.1_3_1}
    1+\alpha\xi+\epsilon y_i^{*}=1+\alpha\xi+\epsilon\frac{\left(\beta-\delta\right)x_2^{*}}{\delta\epsilon}+\epsilon\frac{\left(\beta-\delta\alpha\right)\xi-\delta}{\delta\epsilon}=\frac{(\beta-\delta)x_2^{*}+\beta \xi}{\delta}
\end{equation}

To simplify the denominator expressions in \eqref{trace} and \eqref{det}, we introduce a concise notation $R$ to replace $1+\alpha\xi+x_i^{*}+\epsilon y_2$, which appear in the denominators.

\begin{equation}\label{trace2}
   \begin{split}
Tr_ {(x_i^{*},y_i^{*})}=\frac{1}{\delta^2R^2}\left[\beta\left(x_i^{*}+\xi\right)\left[\beta\left(1-\frac{2x_i^{*}}{k}-\delta\right)\left(x_i^{*}+\xi\right)+\delta^2\left(1+\alpha\xi+x_i^{*}\right)\right]-\delta \left[(\beta-\delta)x_i^{*}+\beta\xi\right]y_i^{*}\right]
 \end{split}
\end{equation}
\begin{equation}\label{det1}
   \begin{split}
   det_ {(x_i^{*},y_i^{*})} = -\frac{\beta (x_i^{*}+\xi)}{\delta^3R^4}\left[\beta^2(x_i^{*}+\xi)^2\left(1-\frac{2x_i^{*}}{k}\right)- \delta\left((\beta-\delta)x_i^{*}+\beta\xi\right)y_i^{*}\right].
   \left[(\beta-\delta)x_i^{*}+(\beta-\delta\alpha)\xi -\delta\right]+\\
   \frac{\beta x_i^{*} y_i^{*}}{\delta R^4}\left[(\beta-\delta)x_i^{*}+\beta\xi-\delta\epsilon\right]\left(1+\alpha\xi+x_i^{*}\right)
   \end{split}
\end{equation}

The theorems \eqref{thm1:chap3_2} and \eqref{secondEq} discuss the stability of coexistence equilibrium states, $E_1$ and $E_2$ respectively.

\begin{theorem}\label{thm1:chap3_2}
The interior equilibrium  $E_1=\left(x_1^*,y_1^{*}\right)$ is saddle if $x_1^{*}$ and $y_1^{*}$ lie within the 
intervals, $x_1^{*}>0$  and \\
  $0<y_1^{*}<\frac{(x_1^{*}+\xi)\left[\beta^2(x_1^{*}+\xi)^2 \left(k-2x_1^{*}\right)\right]\left[(\beta-\delta)x_1^{*}+(\beta-\delta\alpha)\xi -\delta)\right]}{\delta k\left[(x_1^{*}+\xi)\left[ (\beta-\delta)x_1^{*}+\beta\xi\right]\left[(\beta-\delta)x_1^{*}+(\beta-\delta\alpha)\xi -\delta)\right] + \delta x_1^{*}\left[(\beta-\delta)x_1^{*}+\beta\xi-\delta\epsilon\right]\left(1+\alpha\xi+x_1^{*}\right)\right]}$ 
\end{theorem}

\begin{proof}
We simplify the equation in \eqref{det1} to show that $det_ {(x_1^{*},y_1^{*})}< 0 .$ \\
Now, when the parameters and $(x_1^{*},y_1^{*})$ satisfy the condition, \\

$y_1^{*}<\frac{(x_1^{*}+\xi)\left[\beta^2(x_1^{*}+\xi)^2 \left(k-2x_1^{*}\right)\right]\left[(\beta-\delta)x_1^{*}+(\beta-\delta\alpha)\xi -\delta)\right]}{\delta k\left[(x_1^{*}+\xi)\left[ (\beta-\delta)x_1^{*}+\beta\xi\right]\left[(\beta-\delta)x_1^{*}+(\beta-\delta\alpha)\xi -\delta)\right] + \delta x_1^{*}\left[(\beta-\delta)x_1^{*}+\beta\xi-\delta\epsilon\right]\left(1+\alpha\xi+x_1^{*}\right)\right]}$, \\

The determinate at $E_1$ is negative under these conditions. Furthermore, both $x_1^{}$ and $y_1^{}$ must be positive to achieve a positive interior equilibrium. Therefore, the first interior equilibrium point $E_1=\left(x_1^{},y_1^{}\right)$ is a saddle.
\end{proof}


\begin{theorem}\label{secondEq}
The stability of the second interior equilibrium point is studied by dividing the dynamics into following cases.
    $E_2= (x_2^{*}, y_2^{*})$ is, 
\begin{enumerate}[label=(\roman*)]
        \item a stable point or stable focus, if the parameters and the equilibrium satisfy \\
        $y_2^{*}>\frac{(x_2^{*}+\xi)\left[\beta^2(x_2^{*}+\xi)^2 \left(k-2x_2^{*}\right)\right]\left[(\beta-\delta)x_2^{*}+(\beta-\delta\alpha)\xi -\delta)\right]}{\delta k\left[(x_2^{*}+\xi)\left[ (\beta-\delta)x_2^{*}+\beta\xi\right]\left[(\beta-\delta)x_2^{*}+(\beta-\delta\alpha)\xi -\delta)\right] + \delta x_2^{*}\left[(\beta-\delta)x_2^{*}+\beta\xi-\delta\epsilon\right]\left(1+\alpha\xi+x_2^{*}\right)\right]} $ \\ and
        $y_2^*>\frac{\beta\left(x_2^*+\xi\right)\left[\beta\left((1-\delta)k-2x_2^*\right)\left(x_2^*+\xi\right)+k\delta^2\left(1+\alpha\xi+x_2^*\right)\right]}{\delta k\left[(\beta-\delta)x_2^*+\beta\xi\right]}$. 


        \item  a weak focus, if   $y_2^{*}=\frac{(x_2^{*}+\xi)\left[\beta^2(x_2^{*}+\xi)^2 \left(k-2x_2^{*}\right)\right]\left[(\beta-\delta)x_2^{*}+(\beta-\delta\alpha)\xi -\delta)\right]}{\delta k\left[(x_2^{*}+\xi)\left[ (\beta-\delta)x_2^{*}+\beta\xi\right]\left[(\beta-\delta)x_2^{*}+(\beta-\delta\alpha)\xi -\delta)\right] + \delta x_2^{*}\left[(\beta-\delta)x_2^{*}+\beta\xi-\delta\epsilon\right]\left(1+\alpha\xi+x_2^{*}\right)\right]}. $\\
        In this case, there exists one or more than one limit cycles surrounding the point $E_2.$
        \item  an unstable focus or node if  \\
        $0<y_2^{*}<\frac{(x_2^{*}+\xi)\left[\beta^2(x_2^{*}+\xi)^2 \left(k-2x_2^{*}\right)\right]\left[(\beta-\delta)x_2^{*}+(\beta-\delta\alpha)\xi -\delta)\right]}{\delta k\left[(x_2^{*}+\xi)\left[ (\beta-\delta)x_2^{*}+\beta\xi\right]\left[(\beta-\delta)x_2^{*}+(\beta-\delta\alpha)\xi -\delta)\right] + \delta x_2^{*}\left[(\beta-\delta)x_2^{*}+\beta\xi-\delta\epsilon\right]\left(1+\alpha\xi+x_2^{*}\right)\right]}. $\\
        At this point, the majority of the trajectories moving towards the pre extinction point and thereby the prey extinction is almost globally asymptotically stable.

\end{enumerate}
\end{theorem}

\begin{proof}
We evaluate and simplify the trace and the determinant expressions \eqref{trace} and \eqref{det1} at the point $E_2=(x_2^{*},y_2^{*})$.
   \begin{enumerate}[label=(\roman*)]
      \item Under the parametric restriction of\\
$y_2^{*}>\frac{(x_2^{*}+\xi)\left[\beta^2(x_2^{*}+\xi)^2 \left(k-2x_2^{*}\right)\right]\left[(\beta-\delta)x_2^{*}+(\beta-\delta\alpha)\xi -\delta)\right]}{\delta k\left[(x_2^{*}+\xi)\left[ (\beta-\delta)x_2^{*}+\beta\xi\right]\left[(\beta-\delta)x_2^{*}+(\beta-\delta\alpha)\xi -\delta)\right] + \delta x_2^{*}\left[(\beta-\delta)x_2^{*}+\beta\xi-\delta\epsilon\right]\left(1+\alpha\xi+x_2^{*}\right)\right]}$, \hspace{3mm} $det_ {(x_2^{*},y_2^{*})} $ is positive.\\
Also, under the condition, $y_2^*>\frac{\beta\left(x_2^*+\xi\right)\left[\beta\left((1-\delta)k-2x_2^*\right)\left(x_2^*+\xi\right)+k\delta^2\left(1+\alpha\xi+x_2^*\right)\right]}{\delta k\left[(\beta-\delta)x_2^*+\beta\xi\right]},$
$tr_ {(x_2^{*},y_2^{*})}$ is negative. Therefore, the equilibrium point $E_2=(x_2^{*},y_2^{*})$ is an stable point
      \item  If the parameter space satisfy the equality \\
      $y_2^{*}=\frac{(x_2^{*}+\xi)\left[\beta^2(x_2^{*}+\xi)^2 \left(k-2x_2^{*}\right)\right]\left[(\beta-\delta)x_2^{*}+(\beta-\delta\alpha)\xi -\delta)\right]}{\delta k\left[(x_2^{*}+\xi)\left[ (\beta-\delta)x_2^{*}+\beta\xi\right]\left[(\beta-\delta)x_2^{*}+(\beta-\delta\alpha)\xi -\delta)\right] + \delta x_2^{*}\left[(\beta-\delta)x_2^{*}+\beta\xi-\delta\epsilon\right]\left(1+\alpha\xi+x_2^{*}\right)\right]} $, then $det_ {(x_2^{*},y_2^{*})} =0$. Consequently, the equilibrium point $E_2$ is a weak focus. 
        \item  As long as the the parameter values satisfy the inequality\\
        $0<y_2^{*}<\frac{(x_2^{*}+\xi)\left[\beta^2(x_2^{*}+\xi)^2 \left(k-2x_2^{*}\right)\right]\left[(\beta-\delta)x_2^{*}+(\beta-\delta\alpha)\xi -\delta)\right]}{\delta k\left[(x_2^{*}+\xi)\left[ (\beta-\delta)x_2^{*}+\beta\xi\right]\left[(\beta-\delta)x_2^{*}+(\beta-\delta\alpha)\xi -\delta)\right] + \delta x_2^{*}\left[(\beta-\delta)x_2^{*}+\beta\xi-\delta\epsilon\right]\left(1+\alpha\xi+x_2^{*}\right)\right]},$ the $det_ {(x_2^{*},y_2^{*})} <0$. Thus, the equilibrium point $E_2$ is a repeller. 
\end{enumerate}
\end{proof}


\begin{figure}\label{figure3_1}
\centering
\begin{subfigure}{7cm}
\centering\includegraphics[width=7.7cm]{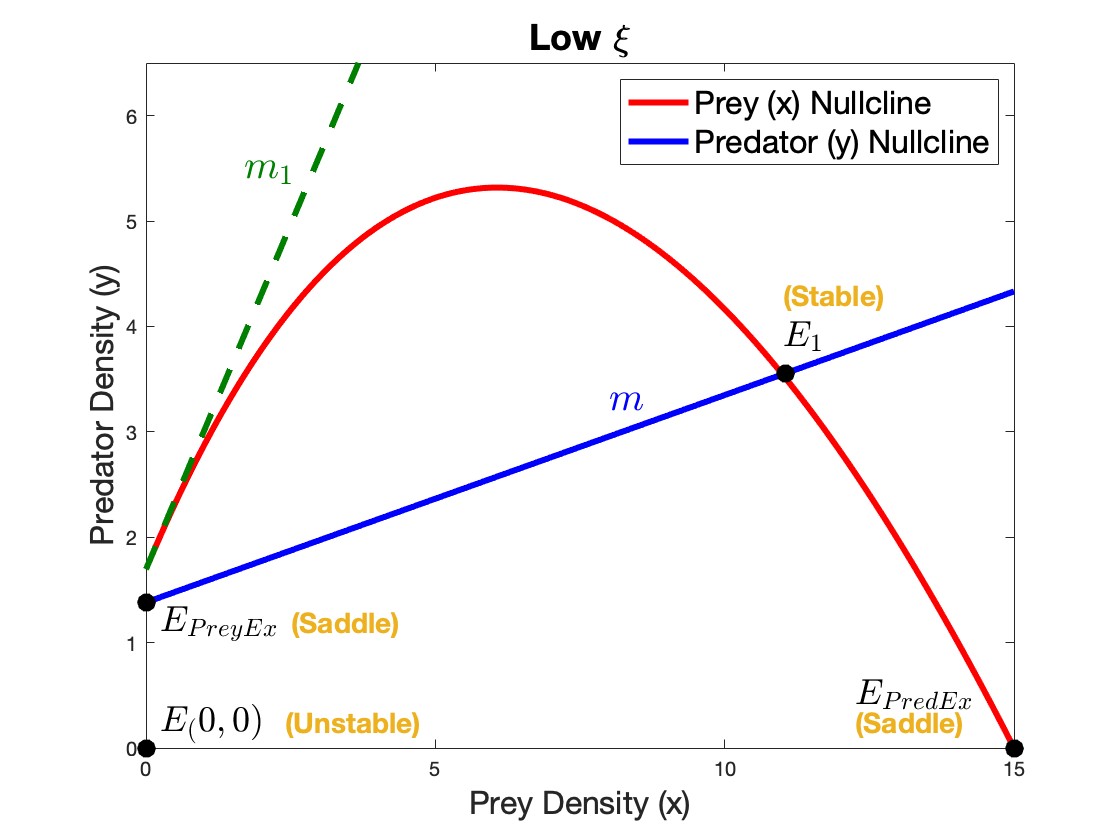}
\caption{} 
\label{cha2Fig1a}
\end{subfigure}%
\begin{subfigure}{7cm}
\centering\includegraphics[width=7.7cm]{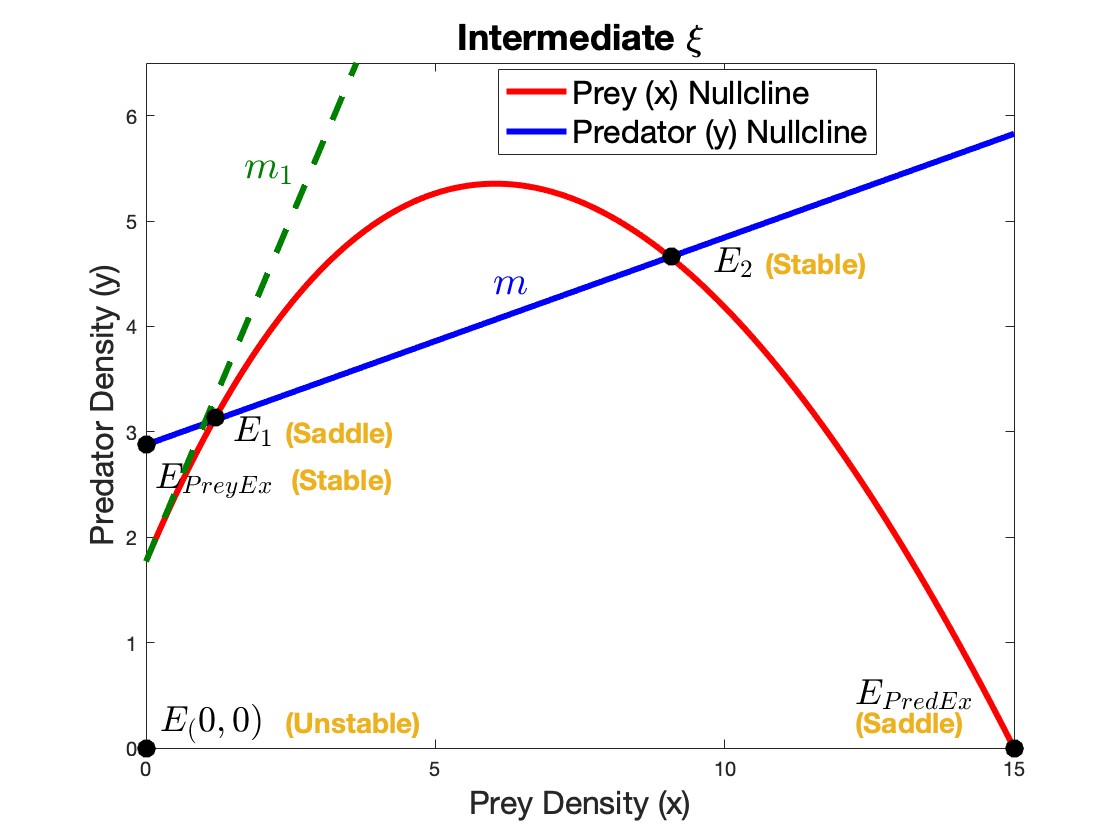}
\caption{}
\label{cha2Fig1b}
\end{subfigure}\vspace{10pt}
 
\begin{subfigure}{7cm}
\centering\includegraphics[width=7.5cm]{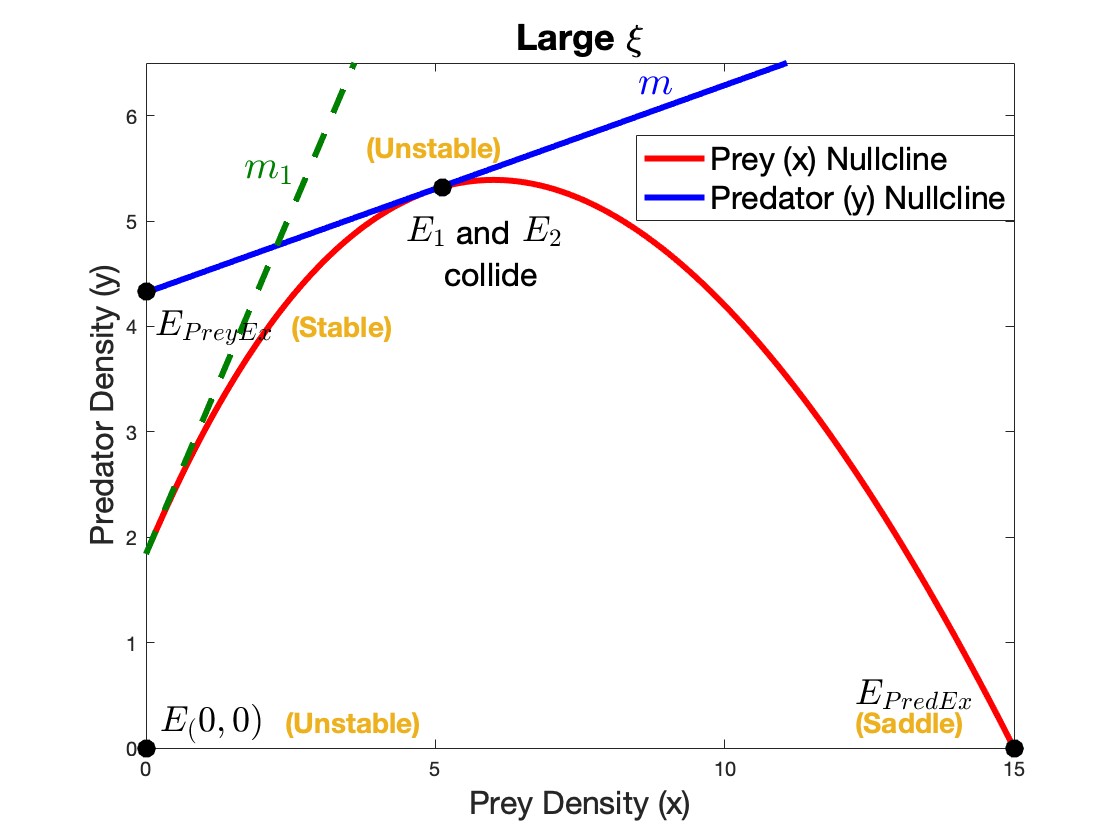}

\caption{}
\label{cha2Fig1c}
\end{subfigure}%
\begin{subfigure}{7cm}
\centering\includegraphics[width=7.7cm]{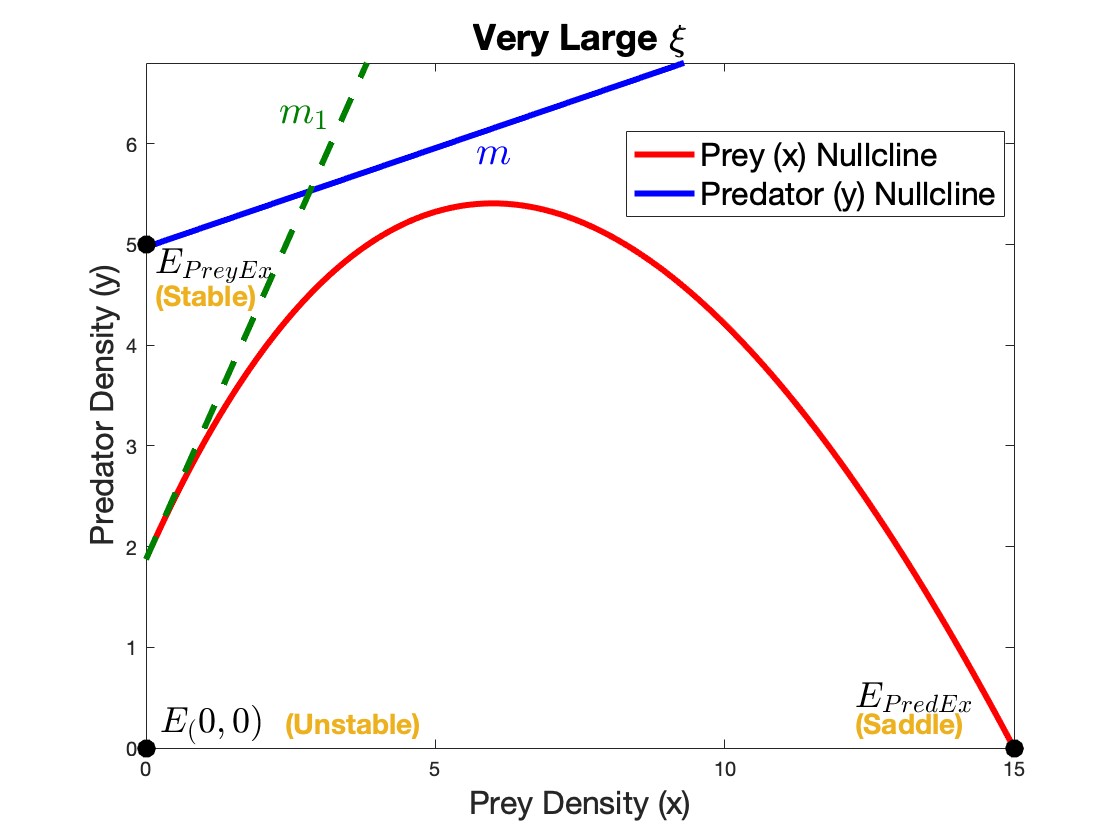}
\caption{}
\label{cha2Fig1d}
\end{subfigure}
\caption{The graphical representation of predator and prey nullclines, along with equilibrium states, is illustrated for different values of $\xi$: ((a)$\xi$=1.6, (b) $\xi$=2, (c)$\xi=2.4828$ (d)$\xi=3$) values by fixing all other ecological parameters, $\alpha=0.1, \beta=0.319, k=15, \epsilon=0.322,$ and $\delta=0.3$. Notably, when $m_1 > m$, the system exhibits two interior equilibrium points for intermediate levels of additional food quantity $(\xi).$ }
\label{figure3_1}
\end{figure}


\begin{figure}\label{figure3_2}
\centering
\begin{subfigure}{7cm}

\centering\includegraphics[width=7.7cm]{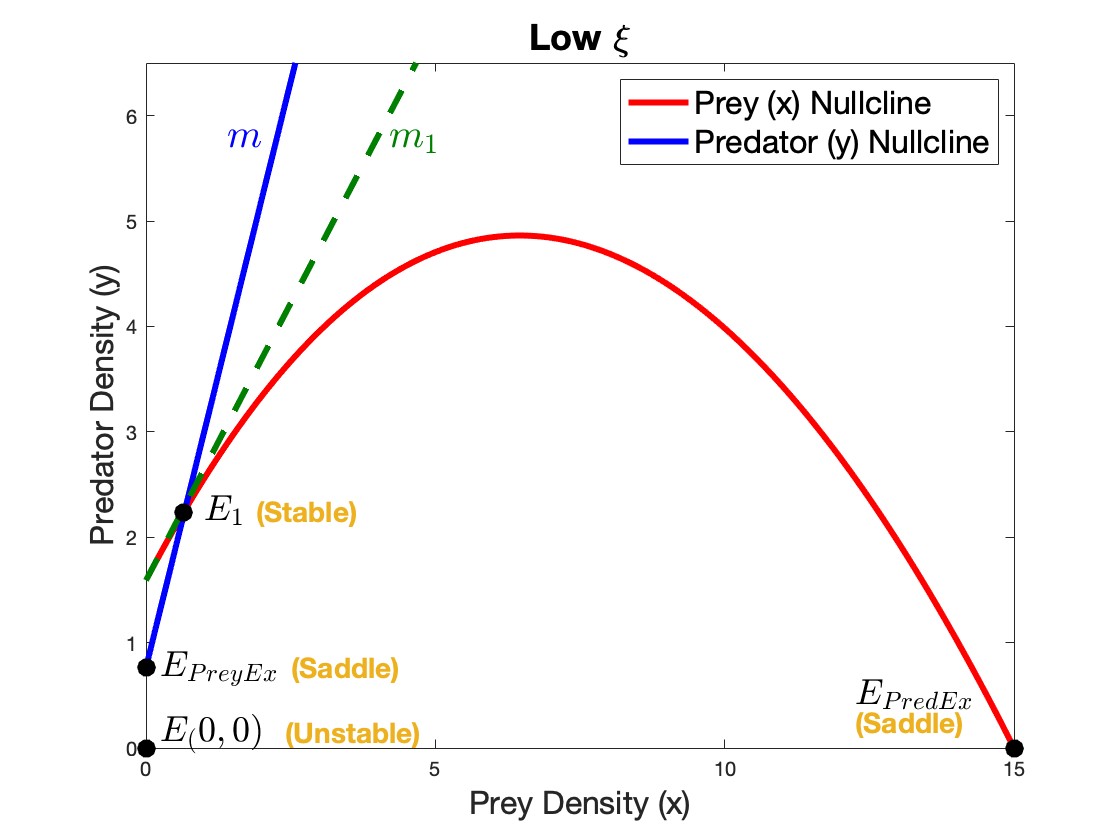}
\caption{} 
\label{cha2Fig2a1}
\end{subfigure}%
\begin{subfigure}{7cm}
\centering\includegraphics[width=7.7cm]{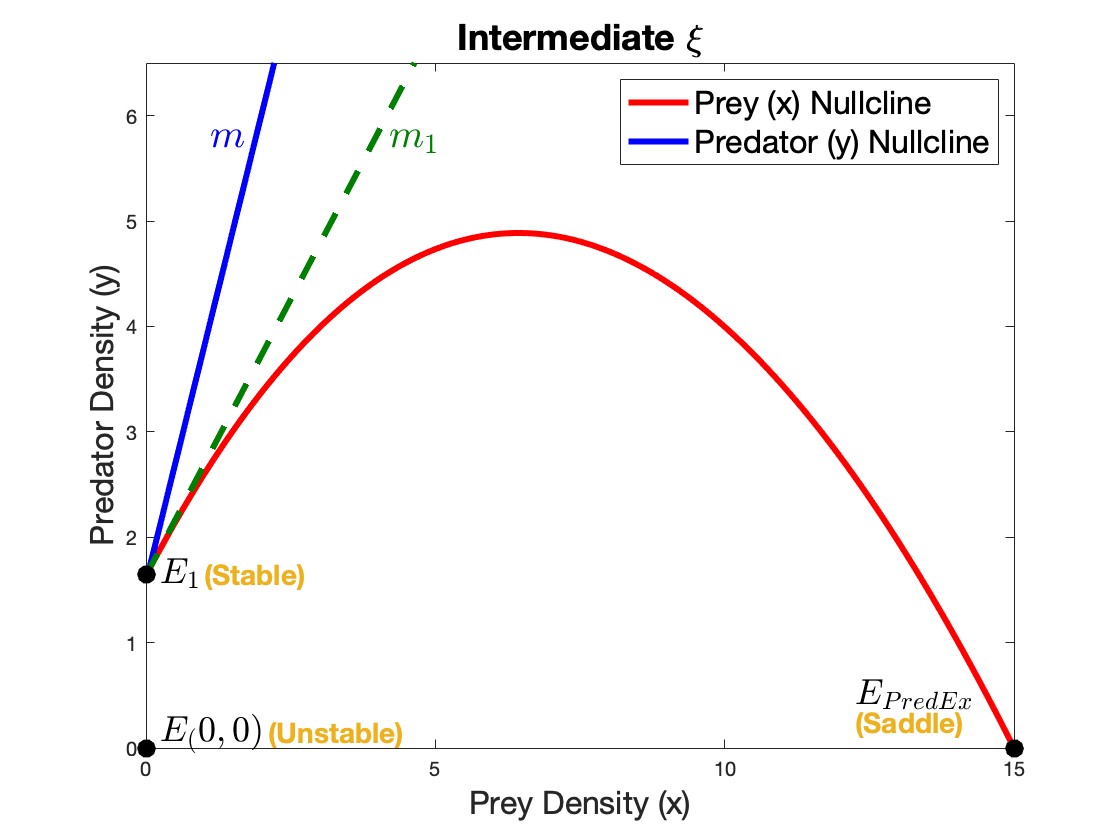}

\caption{}
\label{cha2Fig2b1}
\end{subfigure}\vspace{10pt}
 
\begin{subfigure}{7cm}
\centering\includegraphics[width=7.5cm]{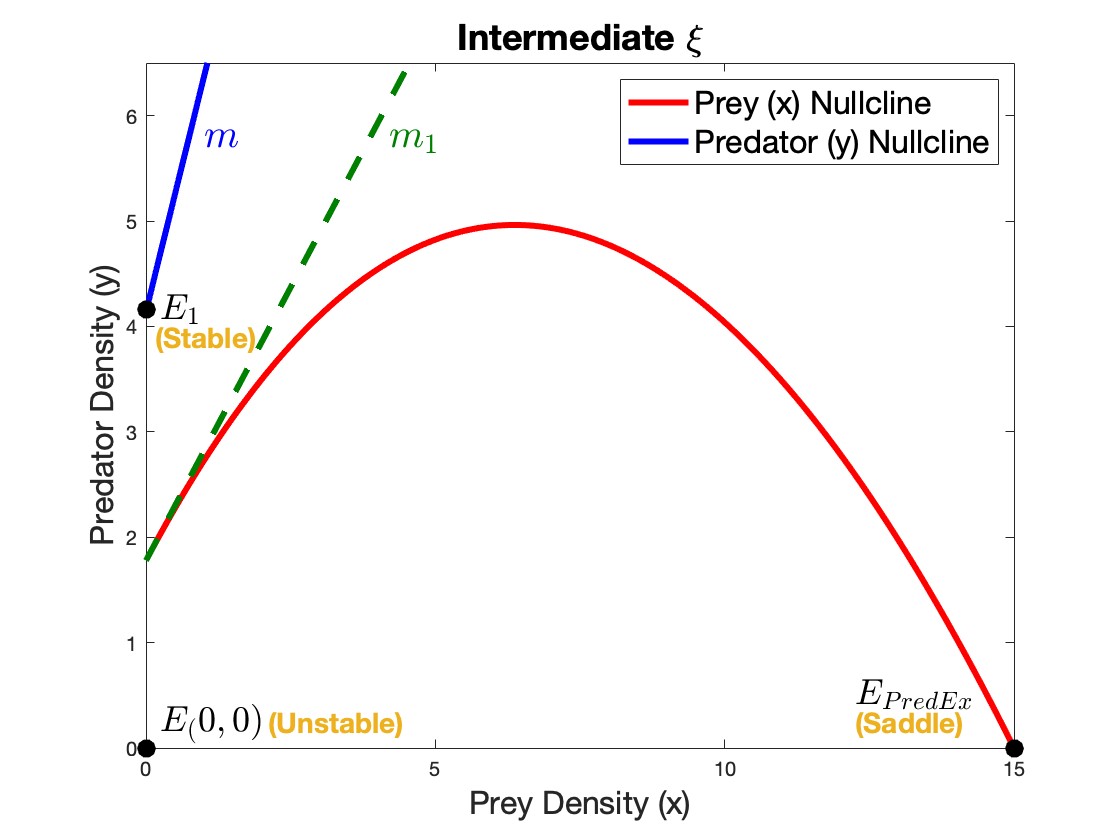}
\label{cha2Fig2c1}
\end{subfigure}\vspace{10pt}
\caption{Illustrates the graphical representation of predator and prey nullclines and the equilibrium states for various  $\xi$: (a) $\xi=0.98$, (b) $\xi=1.1$, and (c) $\xi=1.22$. All other ecological parameters are kept constant: $\alpha=0.32$, $\beta=0.6$, $k=15$, $\epsilon=0.15$, and $\delta=0.45$. In this context, we maintain the parameter constraint $m_1 < m$, implying the system can possess at most only one interior equilibrium point. }
\label{figure3_2}
\end{figure}

\subsection{Behaviour of nullclines of the system}\label{model2NullAna}
Studying the behavior of nullclines that depends on the parameter space enables us to analyze the dynamics of the entire system.
 In this context, the nature of the nullclines and their behavior depends on the parameter values. The additional food is characterized by its quality $(\alpha)$ and quantity $(\xi)$, which can significantly impact the efficacy of pest control. We specifically consider the additional food quantity $(\xi)$ as the parameter of interest to examine the system dynamics. In our analysis, we treat the parameters $k$, $\beta$, $\delta$, $\epsilon$, and $\alpha$ as fixed system parameters, while we vary the control parameter $\xi.$ We will now illustrate the parameter regimes of the different equilibrium states along with their stability concerning the control parameter $\xi.$
It's worth noting that the non-trivial predator nullcline \eqref{eqn:chap3_19} has a slope of $\frac{(\beta - \delta)}{\delta \epsilon}$ and a $y$-intercept of $\frac{(\beta - \delta \alpha)\xi - \delta}{\delta \epsilon}.$ As the control parameter $(\xi)$ is increased gradually, the non-trivial predator nullcline rises. This leads to multiple transitions in terms of the number of interior equilibrium points and their stability. In Figure \eqref{cha2Fig1a}, the dynamics are shown for a lower level of additional food amount $\xi.$ In this scenario, the trivial predator nullcline intersects the trivial prey nullcline at one point, indicating the existence of one interior equilibrium point. As the additional food quantity increases (Figure \ref{cha2Fig1b}), the predator nullcline moves upward, causing it to intersect the prey nullcline at two points, corresponding to the existence of two interior equilibrium points. Continuing to increase the additional food quantity, the two interior equilibrium points move closer and eventually collide with each other, as illustrated in Figure \eqref{cha2Fig1c}. At a significantly high level of additional food quantity, the predator nullcline separates from the prey nullcline, resulting in the absence of interior equilibrium points, as depicted in Figure \eqref{cha2Fig1d}.

\begin{figure}
  \centering
  \includegraphics[width=.45\linewidth]{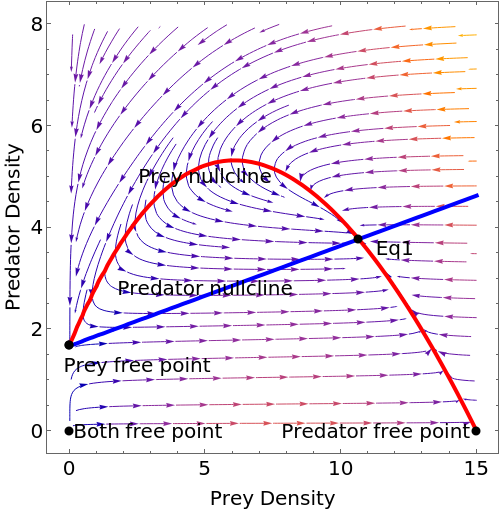} 
\caption{Illustrating the system dynamics across the parameter range of $0 < \xi < 1.65.$ In this scenario, a single stable coexistence state exists. All other ecosystem parameters remain constant: $\beta=0.319, \delta=0.3, k=15, \epsilon=0.322, \alpha=0.1.$}
\label{Bif1}
\end{figure}

\begin{figure}\label{figure4}
\centering
\begin{subfigure}{6cm}
\centering\includegraphics[width=5.8cm]{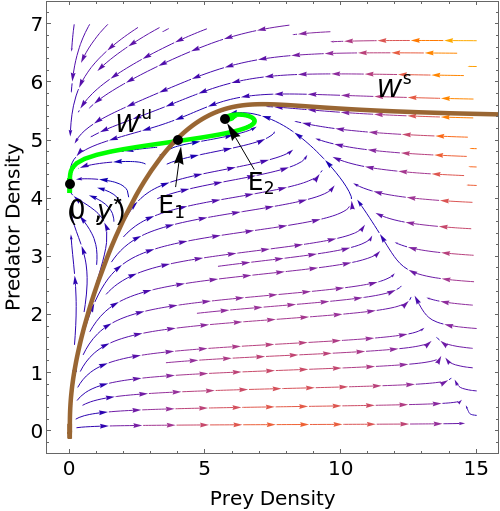}
\caption{} 
\label{figure4_2}
\end{subfigure}
\begin{subfigure}{6cm}
\centering\includegraphics[width=5.8cm]{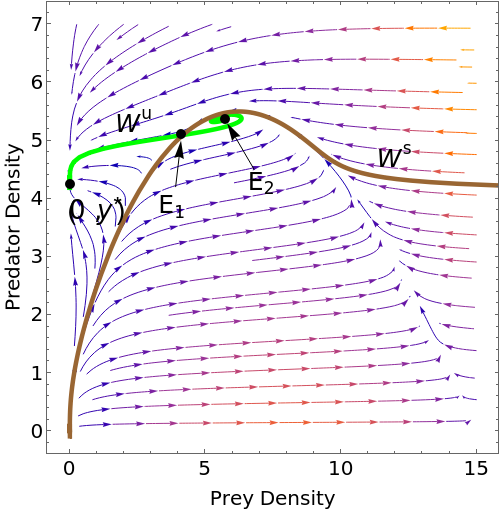}
\caption{}
\label{figure4_3}
\end{subfigure}
\begin{subfigure}{6cm}
\centering\includegraphics[width=5.8cm]{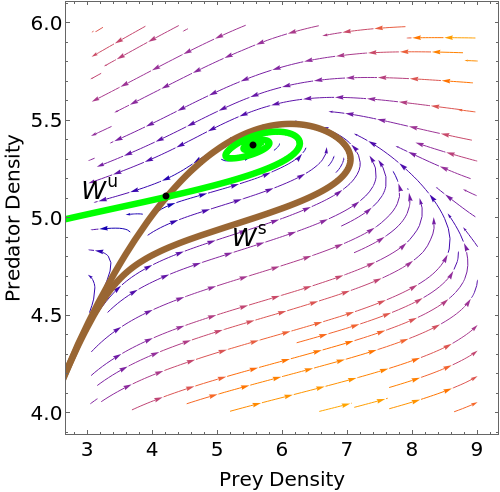}
\caption{}
\label{figure4_4}
\end{subfigure}
\caption{Illustrates the system dynamics within the parameter range of  $1.65<\xi<2.47$. In this scenario, the stable manifold $(W^s)$ and the unstable manifold $(W^u)$ of the interior equilibrium point $E_1$ are in brown and green colors, respectively. Two interior equilibrium points exist, and notably, the system experience bi-stability within this range. Panel (a) $(\xi= 1.68)$ presents the dynamical behavior with predator-prey nullclines and equilibrium, while the panel (b) focuses on the trajectory directions for $\xi=1.92$, and the panel (c) displays the dynamics corresponding to the the parameter value $\xi= 2.469.$ . In all subsequent figures the other fixed ecosystem parameters are $\beta=0.319, \delta=0.3, k=15, \epsilon=0.322,$ and $ \alpha=0.1.$  }
\end{figure}

\section{Numerical Simulations}\label{NumericalSi2}

In this section, we present numerical simulations that validate the results obtained from our analytical guidelines. Specifically, we conduct simulations and generate plots using tools such as Pplane8 (MATLAB continuation package), MATLAB® R2019b, and Mathematica. Our simulations focus on various dynamics, including Hopf, Homoclinic, and Saddle-node bifurcations. Detailed theoretical analyses are provided in the subsequent section (\ref{biAna}). To simulate the system, we vary the key parameter quantity of additional food $(\xi)$ across different ranges, while keeping the ecosystem parameters fixed ($\beta=0.319$, $\delta=0.3$, $k=15$, $\epsilon=0.322$, $\alpha=0.1$). This approach allows us to observe a range of interesting and complex dynamics for different parameter values of $\xi.$ We divide the parameter ranges into distinct cases, each exhibiting unique dynamics and transformations.

\begin{figure}\label{figure5}
\centering
\begin{subfigure}{6cm}
\centering\includegraphics[width=5.8cm]{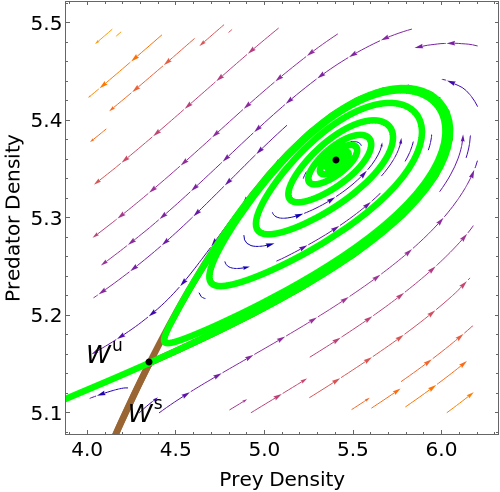}
\caption{} 
\label{figure5_1}
\end{subfigure}%
\begin{subfigure}{6cm}
\centering\includegraphics[width=5.8cm]{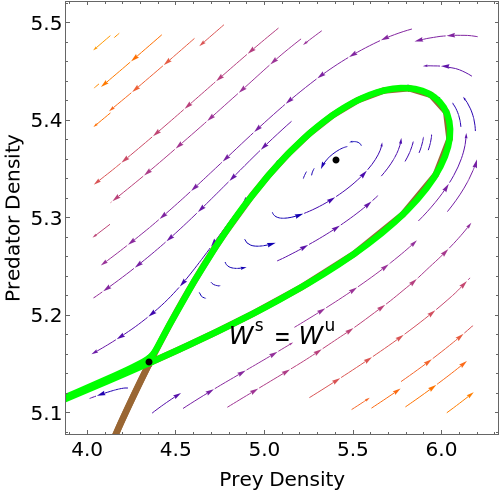}
\caption{}
\label{figure5_2}
\end{subfigure}\vspace{10pt}
 
\begin{subfigure}{6cm}
\centering\includegraphics[width=5.8cm]{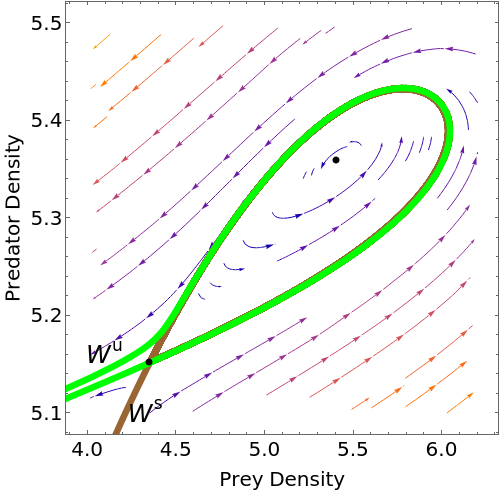}
\caption{}
\label{figure5_3}
\end{subfigure}
\begin{subfigure}{6cm}
\centering\includegraphics[width=5.8cm]{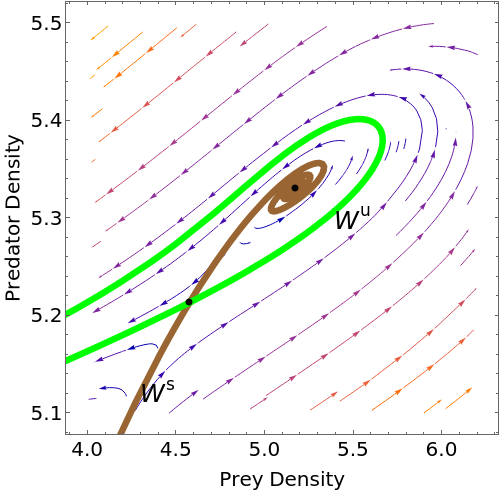}
\caption{}
\label{figure5_4}
\end{subfigure}
\caption{Here we explore the impact of varying levels of additional food  on inducing a Homoclinic bifurcation. The transitions of the stable manifold $W^{s}(E_1)$ and the unstable manifold $W^{u}(E_1)$ of the saddle point $E_1$ are illustrat for different additional food quantities: (a) ($\xi=2.4741312.$) , (b) ($\xi=2.4741313.$) , (c)($\xi=2.4741314.$)  and (d)($\xi=2.475$). A Homoclinic loop forms when  $W^{s}=W^{u}$ at $\xi=2.4741313.$ as in Figure (b). The ecosystem parameters remain constant: $\beta=0.319, \delta=0.3, k=15, \epsilon=0.322, \alpha=0.1.$}
\end{figure}

\begin{figure}
\label{Bif4}
\centering
\begin{subfigure}{6cm}
\centering\includegraphics[width=5.8cm]{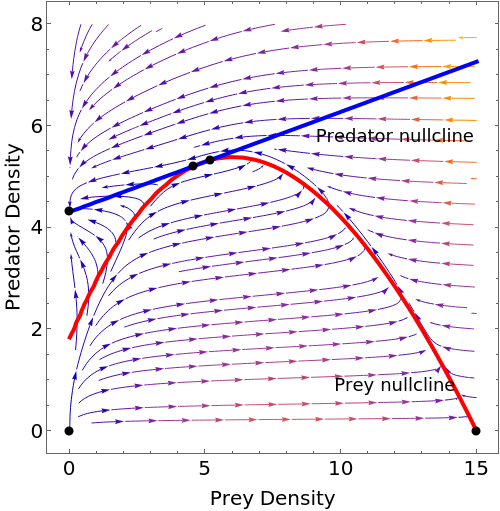}
\caption{} 
\label{Bif4_1}
\end{subfigure}%
\begin{subfigure}{6cm}
\centering\includegraphics[width=5.8cm]{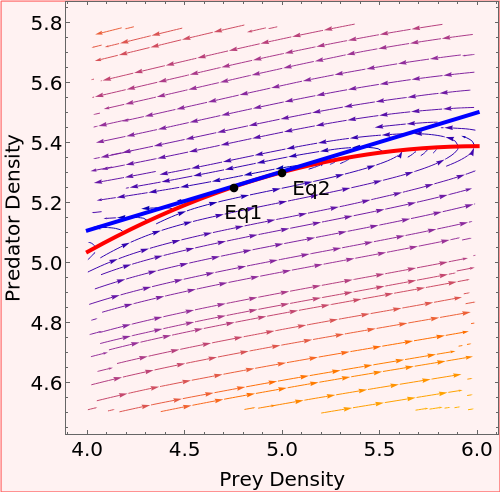}
\caption{}
 \label{Bif4_2}
\end{subfigure}\vspace{10pt}
\begin{subfigure}{6cm}
\centering\includegraphics[width=5.8cm]{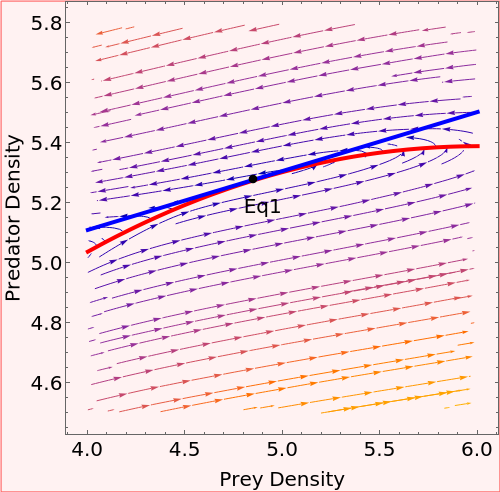}
\caption{}
 \label{Bif4_3}
\end{subfigure}
\begin{subfigure}{6cm}
\centering\includegraphics[width=5.8cm]{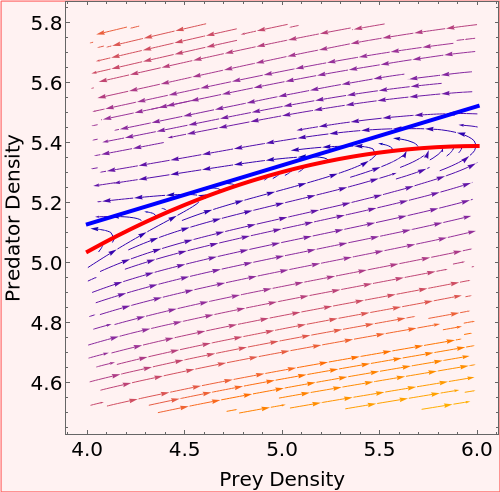}
\caption{}
 \label{Bif4_4}
\end{subfigure}
\caption{Numerical illustration depict the behavior of coexistence states in the presence of additional food with various supply levels. Figures (a) and (b) display the dynamic corresponding to the parameter value $\xi=2.4827$. In Figure (b), the dynamics near the two interior equilibrium points of Figure (a) are magnified. Figures (b), (c), and (d) clearly demonstrate the transitions of several interior equilibrium points. Figure (c) presents the collision of equilibrium points at $\xi=2.4828$, and Figure (d) corresponds to the departure of the predator nullcline from the prey nullcline at $\xi=2.4829$ leading to the absence of any coexistence states. The ecosystem parameters remain constant: $\beta=0.319, \delta=0.3, k=15, \epsilon=0.322, \alpha=0.1$.}
\label{Bif4}
\end{figure}

\textbf{Case 1:}
In the range of $0 < \xi < 1.65$, the system exhibits only one interior equilibrium point, which is locally stable, as shown in Figure \ref{Bif1}. All other trivial and axial equilibrium points, including the prey-free equilibrium $(0, y^*)$, are unstable. From a biological perspective, if the predator is provided with a low level of additional food, both predator and prey populations are able to sustain. Furthermore, the predator density in the ecosystem exceeds the prey. \\
\textbf{Case 2:}
Within the interval $1.65 < \xi < 2.47$, two interior equilibrium points are observed, resulting from the non-trivial predator and prey nullclines intersecting twice. The larger interior equilibrium point ($E_2$) is stable, while the other interior point ($E_1$) is a saddle-node. The main observation in this parameter regime is the existence of two stability regimes: the prey-free region and the stable second interior point ($E_2$) region. In this context, the behavior of the stable manifold ($W^{s}(E_1)$) acts as a separatrix, dividing the entire space into these two stable regimes. As seen in Figures \ref{figure4_2} and \ref{figure4_3}, trajectories starting below the stable manifold $(W^{s}(E_1))$ move towards the prey-free state $(0, y^)$, whereas those starting above it tend to converge towards the second interior equilibrium point $E_2(x_2^{}, y_2^{*})$. Consequently, the system exhibits bi-stability (multi-stability) due to the presence of two stable regions. From an ecological standpoint, this scenario, which heavily depends on the initial predator and prey densities in the ecosystem, can be interpreted as either the survival of both species or the complete elimination of the prey species. Remark that as $\xi$ is slightly elevated, the stability region of the prey-free equilibrium broadens.\\
\textbf{Case 3:}
While the system dynamics remains  bi-stable, a sequence of unstable limit cycles centered around the interior equilibrium point $E_2(x_2^{*},y_2^{*})=(8.02768 ,5.05513 )$ emerges at $\xi=2.2.$ 
The computed eigenvalues and eigen vectors of $E_2$ are\\
$-0.118487+0.011339i$, $-0.118487-0.011339i$ and $\left[0.993038 ,-0.116419-0.0179396i\right]$ and $\left[0.993038 ,\\
-0.116419-0.0179396i\right]$ respectively. The system clearly undergoes a Hopf bifurcation at this parametric value.\\
\textbf{Case 4:}
At the value $\xi= 2.469 $, the stable manifold of $E_1$ starts to twist around the interior equilibrium points: $E_1(x_1^{*},y_1^{*})=(4.12491,5.08031)$ and $E_2(x_2^{*},y_2^{*})=(5.6355, 5.37743 )$, as depicted in Figure \ref{figure4_4}. Consequently, the stability region of the prey free point expands, leading to the inevitability of prey extinction for a majority of initial prey and predator values. In other words, the prey extinction point tends to attract trajectories from most of the feasible initial points within the system.\\
\textbf{Case 5:}
In this case, we concentrate on the possibility of homoclinic bifurcation. Figure \ref{figure5_1} illustrates the dynamics just before the occurrence of homoclinic bifurcation, where the unstable manifold $W^{u}(E_1)$  lies entirely inside the stable manifold $W^{s}(E_1)$ orbit. As the parameter value is steadily increased, the limit cycle of $E_2=(5.40245,5.35891)$ collides with the interior saddle $E_1=(4.34883,5.15167)$ at the parameter value $\xi=2.4741313 $ depicted in Figure \ref{figure5_2}. At this point, the stable manifold $W^{s}(E_1)$ and the unstable manifold $W^{u}(E_1)$ collide with the limit cycle simultaneously. As consequence, the system undergoes a homoclinic bifurcation at $\xi=2.4741313$. Figure \ref{figure5_3} illustrates the system dynamics just after the unstable manifold $W^{u}(E_1)$ leaves the homoclinic orbit. For relatively large parameter value $(\xi=2.475)$, the system reacts by shrinking the stable manifold towards the second interior equilibrium point $E_2$  as shown in Figure \ref{figure5_4}. \\
\textbf{Case 6:}
At the parameter value $\xi=2.478$, the eigen values and eigen vectors of $E_2(x_2^{*},y_2^{*})=(5.26541,5.34353)$ are found to be $0.000645064+0.0471803i, 0.000645064-0.0471803i,$ and $\left[0.988031,0.123707-0.0921525i\right],$ 
$\left[0.988031,0.123707+0.0921525i\right]$ respectively. In this situation, nearly all the trajectories converge toward the prey-free equilibrium state. With the continuous increase of $\xi,$ the two interior equilibrium move closer to each other and eventually, collide at the $\xi$ value 2.4828. At this point, the system experience the saddle-node bifurcation as shown in Figure \ref{Bif4}. Specifically, Figure \ref{Bif4_2} displays the dynamic around the two interior points just before the collision and Figure \ref{Bif4_3} depicts the behavior of the system after the collision. Furthermore, Figure \ref{Bif4_4}, demonstrates the predator nullcline departing from the prey nullcline at the parameter value 2.4829. In this scenario, the system no longer has any interior equilibrium state. For parameter values that exceed this threshold, the prey-free equilibrium solution tends to be globally asymptotically stable based on our numerical simulations. Interestingly, when the predator is provided with a significantly larger additional food quantity ($\xi>2.478$), the prey species is driven to complete extinction.
\section{Local Bifurcation Analysis}\label{biAna}
We employ bifurcation theory to examine changes in the qualitative behavior of the equilibrium solutions as the control parameter $\xi$ is altered while keeping all other parameters constant. The system exhibits variety of bifurcation phenomena, including Hopf bifurcation, Saddle-node bifurcation and Homoclinic bifurcation. When the system possesses two interior equilibrium points, the interior point $E_1$ is always a saddle and the behaviour of its manifolds ( $(W^{s}(E_1))$ and $(W^{u}(E_1))$) significantly influences the system dynamics.
\subsection{Hopf bifurcation}
We utilize the Hopf bifurcation theorem \ref{thmchap3_1} to construct the theorem \ref{thmchap3_2}, which proposes that the system \eqref{model1} undergoes Hopf bifurcation under certain parametric constraints.

\begin{theorem}\label{thmchap3_1}{Hopf Bifurcation Theorem}\cite{S26}\\
If $b\left(\xi\right)$ and $c\left(\xi\right)$ are the smooth functions of $\xi$ in an open interval about $\xi^{*}\in \mathbb{R}.$ such that characteristic equation has a pair of imaginary eigenvalues $\lambda_1=p\left(\xi\right)+iq\left(\xi\right)$ and $\lambda_2=p\left(\xi\right)-iq\left(\xi\right)$ with $p,q\in \mathbb{R}.$ so that they become purely imaginary of  $\xi=\xi^{*}$ and $\frac{dp}{d\xi}|_{\xi=\xi^{*}}	\neq 0$, then a Hopf bifurcation occurs around the equilibrium point, $E=\left( x^{*},y^{*}\right)$ at $\xi=\xi^{*}.$
That is, stability of $E=\left(x^{*},y^{*}\right)$  changes accompanied by the certain of a limit cycle at $\xi=\xi^{*}.$
\end{theorem}
\begin{theorem}\label{thmchap3_2}
 The system undergoes a Hopf-bifurcation with respect to bifurcation parameter $\xi$ around the equilibrium point $E_2=\left(x_2^{*},y_2^{*}\right)$ if 
 \begin{enumerate}
     \item  $y_2^{*}\neq\frac{\beta(x_2^*+\xi)}{\alpha}\frac{2\alpha\delta(1+\alpha+x_2^*)-\beta(\alpha x_2^*+x_2^*+2\alpha+1)}{(\beta-2\delta)x_2^*+\beta\xi}$
     \item $\xi< 1+\alpha\xi+\epsilon+ y_2^{*}$
 \end{enumerate}
\end{theorem}
The proof of the theorem is attached in the section \ref{ProveHoph} in Appendix.
\vspace{3mm}
\subsection{Saddle node bifurcation} \label{saddleNode}
The system ~\eqref{model1} exibits Saddle-node bifurcation when the two interior equilibrium points, $E_1=\left(x_1^{*},y_1^{*}\right)$ and $E_2=\left(x_2^{*},y_2^{*}\right)$, collide with each other. We formulate Theorem \ref{thmchap3_11} to outline  the conditions that the system parameters must satisfy in order to system undergo saddle-node bifurcation. We begin by stating Sotomayor's Theorem \ref{thmchap3_10}, which is employed to derive Theorem \ref{thmchap3_11}."

\begin{theorem}\label{thmchap3_10}
Consider the following system
\begin{equation}
    x^.=F(x,\alpha)
\end{equation}
Where $\alpha \in \mathbb{R},$ a parameter with $\alpha_0$ being the bifurcation threshold. Suppose that $F=(x,\alpha)=0$ and the  matrix $A=DF(x_0;\alpha_0)$ has a simple eigen value $\lambda=0$ with eigen vector $U$ and that $A^T$ has an eigen vector $W$ corresponding to the eigen value $\lambda=0$. Furthermore, suppose that $A$ has $k$ eigen values with negative real part and $(n-k-1)$ eigen values with positive real part and that the following conditions are satisfied.
\begin{equation}\label{thmchap3_33}
    W^{T} F_{\alpha}(x_0;\alpha_0)\neq 0,
\end{equation}

\begin{equation}\label{thmchap3_4}
     W^{T} \left[D^2F_{\alpha}(x_0;\alpha_0)(U;U)\right] \neq 0
\end{equation}

then the system experiences saddle-node bifurcation at the equilibrium point $x_0$ as the parameter $\alpha$ passes through the bifurcation value $\alpha=\alpha_0$ 
\end{theorem}

\begin{theorem}\label{thmchap3_11}
The system ~\eqref{model1} undergoes a Saddle–node bifurcation around $\Tilde{E}=\left(\Tilde{x},\Tilde{y}\right)$ with respect to the bifurcation parameter $\xi=\Tilde{\xi}$, if the parameters and equilibrium points satisfy following conditions:

\begin{itemize}
   \item $\Tilde{y}\neq\frac{(\beta-\delta)\Tilde{x}+(\beta-\delta\alpha)\xi-\delta}{(1+\alpha \xi+\Tilde{x})(\beta-\delta\alpha)}$
    \item $k>2\Tilde{x}$ and $\epsilon<1$
   \item $\frac{(\beta-\delta)\Tilde{x}+\beta\xi}{(\beta-\delta)(\Tilde{x}+\xi)}<\Tilde{y}<\frac{\beta^2(k-2\Tilde{x})(\Tilde{x}+\xi)^2}{\delta k\left[(\beta-\delta)\Tilde{x}+\beta\xi\right]}$
     \item$\frac{\left(1+\alpha\right)\xi-1}{\left(2-\epsilon\right)\epsilon}<\Tilde{y}<\frac{\beta\epsilon(\Tilde{x}+\xi)(1-\delta\epsilon)-\delta\Tilde{x}}{\delta\epsilon}$
\end{itemize}
\end{theorem}
\vspace{4mm}

\section{Global Bifurcation Analysis}
\subsection{Homoclinic Bifurcation}
\subsubsection{The topologically equivalent system}
To provide the theoretical validation of Homoclinic Bifurcation, the system \eqref{model1} is transformed into a specific form by following the methodology presented in \cite{S20}, \cite{S21}, \cite{S23}, \cite{S25} that have used a specific transformation to simplify the system into a topological equivalent form. In this methodology, a diffeomorphism  is defined to normalize the system and to make it simple.

\begin{definition}
[\textbf{Diffeomorphism}]: A \emph{diffeomorphism} is a $C^{1}$ bijective mapping $f:M \mapsto N$ of a $C^{1}$ manifold M ( e.g. of a domain in a Euclidean space) into a $C^{1}$ manifold N for which $f^{-1} \in C^{1}(M)$. If $f(M)=N$, one says that $M$ and $N$ are {diffeomorphic}. If $f,f^{-1}$ are weakened such that $f,f^{-1} \in C^{0}(M)$, then $f$ is said to be a \emph{homeomorphism}.
\end{definition} 

\begin{definition}
[\textbf{Equivalence}]: Two flows $(\mathbb{R}^{n},f_{t})$ and $(\mathbb{R}^{n}, g_{t})$ are equivalent if there exists a bijection $h:\mathbb{R}^{n} \mapsto \mathbb{R}^{n}$ such that for every $t \in \mathbb{R}$ that $h \circ f_{t} = g_{t} \circ h$.
If $h$ is a \emph{diffeomorphism} then $f_{t}$ and $g_{t}$ are \emph{differentiably equivalent}. If If $h$ is a \emph{homeomorphism} then $f_{t}$ and $g_{t}$ are \emph{topologically equivalent}.
If two systems are \emph{differentiably equivalent} then they are \emph{topologically equivalent}.
\end{definition}

\begin{lemma}
The system ~\eqref{model1} is topologically equivalent to the polynomial system given by \eqref{eq:top1}.

System \eqref{model1} is topologically equivalent to the polynomial system given by \\
\begin{equation}\label{eq:top1}
  \begin{cases}
   \frac{du}{d\tau}= u\left[k\left(1-u\right).(P+u+Qv)-v \right]\\
     
     \frac{dv}{d\tau}= v\left[R \left(u+M\right)-N(P+u+Q v) \right]
    \end{cases}   
\end{equation}
where P=$\frac{1+\alpha \xi}{k}, Q=\frac{\epsilon}{k}, R=\frac{\beta}{k},N=k\delta$ and $M=\frac{\xi}{k}$
\end{lemma}

\begin{proof}

We define the solution set as,
\begin{equation}
    \Bar{\phi}=\{(u,v)\in \mathbb{R}^2 \hspace{2mm} \vert \hspace{2mm} u \ge 0, v\ge 0 \}
\end{equation}

Now, we use the transformations, $x=ku$ and $y=v.$ The transformed system is \\

\begin{equation}
\label{top1}
\begin{cases}
  \frac{du}{dt}=u\left(1-u\right)-\frac{u.v}{1+\alpha \xi+ku+\epsilon v} \\
  \frac{dv}{dt}=\frac{\beta (ku+\xi)v}{1+\alpha \xi +ku+\epsilon v} - \delta v 
  \end{cases}
\end{equation}

Finally, we re-scale the time parameter preserving the time orientation by $t=\frac{1}{P+u+Qv}\tau.$ to system \eqref{top1} to obtain the following topological equivalent system.

\begin{equation}\label{eq:top1}
  \begin{cases}
   \frac{du}{d\tau}= u\left[k\left(1-u\right)(P+u+Qv)-v \right]\\
     
     \frac{dv}{d\tau}= v\left[R \left(u+M\right)-N(P+u+Q v) \right]
    \end{cases}  \nonumber  
\end{equation}
where P=$\frac{1+\alpha \xi}{k}, Q=\frac{\epsilon}{k}, R=\frac{\beta}{k},N=k\delta$ and $M=\frac{\xi}{k}.$

So, according to the transformations used to get system \eqref{eq:top1},
 $\psi: \Bar{\phi} \times \mathbb{R} \longrightarrow \phi \times \mathbb{R} $\\

 $\psi(u, v,\tau)= \left(ku, v, (P+ u+Q v)\tau  \right)=(x, y, t)$\\

 $\psi$ can be easily verified to be a diffeomorphism. Herein, the Jacobian matrix of the diffeomorphism $\psi(u, v,\tau)$ is,\\
 
$J_{\psi(u, v,\tau)}=\begin{bmatrix}
k & 0 & 0
   \\
0 & 1 & 0 \\
\tau & \epsilon\tau  & (P+u+\epsilon v)
 \end{bmatrix}$ with the determinant\\
 
 $Det\left(J_{\psi(u, v,\tau)}\right)= k(P+ u+Q v)>0$ as $(u,v) \in \Bar{\phi}.$ So, $\psi$ is a diffeomorphism.

\end{proof}

\begin{lemma}
    \label{boundedness}
    The solutions of the system \eqref{model1} are bounded if $\epsilon(1-\delta)<1.$
\end{lemma}

\begin{proof}
    To show the boundedness of system \eqref{model1} we study the topologically equivalent system \eqref{eq:top1} which offers the same dynamics as the original system.\\
    We know, $u$ is bounded as $\frac{du}{d\tau}<0$ so the trajectories of $u$ remain bounded.\\
    In order to show system \eqref{eq:top1} is bounded we need to show $(0,\infty)$ is unstable.\\
    We apply Poincare compactification using transformation $(u,v,\tau)=(\frac{X}{Y},\frac{1}{Y},Y^3T).$\\
    The transformed system is given by,
    \begin{equation} \label{bound1}
       X= \begin{cases}
            \frac{dX}{dT} = X(-Y-k(X-Y)(X+PY+Q)+Y^3(N(Q+X+PY)-R(X+MY)))\\
            \frac{dY}{dT} = Y(N(Q+X+PY) - R(X+MY))
        \end{cases}
    \end{equation}
  The Jacobian of system $\eqref{bound1}$ at origin is \\
$DX(0,0)= \begin{bmatrix}
    0 & 0 \\
    0 & 0
\end{bmatrix}$ \\
and the origin is a non-hyperbolic singularity. To desingularize the origin, we consider the directional blowing-up technique (\cite{S50}, \cite{S51}, \cite{S52}, \cite{S23}).\\
We use the transformation, $X=rw, Y=w$ and $t=wT$. The new transformed system is \\
\begin{equation} \label{bound2}
   \Bar{X}= \begin{cases}
        \frac{dr}{dt}=r((k+1)(1-R)(Q+(P+R)w)+(w^2-1)(N(Q+(P+R)w)-Rw(R+M))) \\
        \frac{dw}{dt}=w(N(Q+(P+R)w)-Rw(M+R))
    \end{cases}
\end{equation}
  The Jacobian of system $\eqref{bound2}$ at origin is \\
$D\Bar{X}(0,0)= \begin{bmatrix}
    -1+kQ-NQ & 0 \\
    0 & NQ
\end{bmatrix}$ \\

So if $detD\Bar{X}=kQ-NQ-1=\epsilon-\epsilon \delta-1<0$ then $(0,0)$ is less than zero then $(0,0)$ is a saddle point of the vector field $X$ and of $\Bar{X}$, then point $(0,\infty)$ is a saddle point of the compactified vector field of the original system \eqref{model1}. Hence, the solutions of the system are bounded.

\end{proof}

We present several further definitions next.

\begin{definition}
[\textbf{$\omega$-limit set}] The $\omega-$limit set of a point $x_0$ is the set \\
$\omega(x_0)=\{$ x: there exists an unbounded, increasing sequence $\{t_k\}$ such that $\lim_{k\to \infty}{F(t_k,x_0)}=x\}$.\\
\end{definition}

\begin{definition}
[\textbf{$\alpha-$limit set}] The $\alpha-$limit set of a point $x_0$ is the set \\
$\alpha(x_0)=\{$ x: there exists an unbounded, decreasing sequence $\{t_k\}$ such that $\lim_{k\to \infty}{F(t_k,x_0)}=x\}$.\\
\end{definition}

\begin{definition}
[\textbf{Heteroclinic and Homoclinic Orbit}] An orbit $\gamma$ that connects two different equilibrium points
$x_1$, $x_2$, i.e. such that $\alpha(\gamma) = x_1$ and $\omega(\gamma) = x_2$ is called a heteroclinic orbit. A non trivial orbit such that $x_1 = x_2$ is called a homoclinic orbit (or loop)
\end{definition}


\begin{theorem}
\label{homoclinic}
For a fixed parameter set $P^*(k^*,\alpha^*,\epsilon^*,\delta^*,\beta^*)$ when the system is bounded i.e. $\epsilon^*(1-\delta^*)<1$  there exists $\xi=\xi^*$ for which a homoclinic orbit exists.
\end{theorem}

\begin{proof}
    When $\epsilon^*(1-\delta^*)<1$ the system is bounded as in lemma \ref{boundedness}. Let $\Gamma$ be the positive invariant set of system \eqref{eq:top1}. \\
    Let $P_1(u_1,v_1)$ and $P_2(u_2,v_2)$ be the two interior equilibrium points. \\
    According to theorem \ref{thm1:chap3_2}, for the choice of the parametric set $P^*$, the interior equilibrium point $P_1(u_1,u_2)$ is a saddle point. So there exist two manifolds, stable and unstable, passing through $P_1$.\\
    Let $W^u_+(P_1)$ be the right unstable manifold of $P_1$ and $W^s_+$ be the superior stable manifold of $P_1$.\\
    For the given parametric set, we see that for some $\xi=\xi_1$, $P_2$ is an unstable node with a surrounding stable limit cycle as seen in Figure \ref{figure5_1}.\\
    If the parameter is increased to $\xi=\xi_2>\xi_1$, then the point $P_2$ becomes a stable node as seen in Figure \ref{figure5_4}.\\
    As the system is bounded, we can assume the $\omega-limit$ of $W^u_+(P_1)$ remains in $\Gamma$ for the fixed parameter set. Thus the trajectory is bounded. 
    For $\xi=\xi_1$ ($P_2$ is unstable) the stable manifold $W^s_+(P_1)$ surrounds the unstable manifold $W^u_+(P_1)$.\\
    Similarly, for $\xi=\xi_2$ ($P_2$ is stable) the stable manifold $W^s_+(P_1)$ is enclosed by the unstable manifold $W^u_+(P_1)$.\\
    Then, by the Existence and Uniqueness theorem \cite{S50} of solutions, there exists $\xi=\xi^*$ such that $\xi_1<\xi^*<\xi_2$ where the stable manifold $W^s_+(P_1)$ and the unstable manifold $W^u_+(P_1)$ are the same thus forming a homoclinic orbit \cite{S54} as seen in Figure \ref{figure5_2}.
    
\end{proof}

\section{Conclusion and Discussion}
There have been many studies carried out to explore the use of AF to effectively eradicate unnecessary pest and pathogens that are harmful for ecosystems \cite{S02}, \cite{S35},\cite{S36},\cite{S37},\cite{S11},\cite{S12}, \cite{S40}. Some of these models lead to unbounded growth of the introduced predator, in the case of pest eradication. One method of avoiding this while providing an AF source, is incorporating intra-species interactions such as mutual interference. This was carried out in \cite{S02}. We have further explored the dynamical behavior of the model proposed in \cite{S02}, wherein it was shown that the system can have various interesting dynamics with the occurrence of only one coexistence state. In the current work we show the existence of two coexistence states under certain restrictions of AF quantity $(\xi)$, that causes novel dynamics including different bifurcation scenarios.  The AF quantity is considered to be the main control parameter in our study. The results obtained show that the system undergoes bi-stability for a certain parameter range of AF quantity in which the stability tremendously depends on the initial prey and predator densities. To clarify this, we observed that one coexistence state is locally attracting for certain initial prey/predator data, while the rest of initial values are attracted towards the pest-free state. The division of the phase into these two regimes is via a separatrix, which is essentially the stable manifold of the other coexistence state, which is a saddle.
In this context, the stable manifold of the first interior point separate the predator-prey domain into two major regions, the prey free stability region and the coexistence stability region where both species can survive. There are significant benefits of this scenario because the stability region of pest-free state can be enlarged with small variation in $\xi.$  

The system also experiences local bifurcations such as hopf bifurcation, and saddle node bifurcation, as well as global homoclinic bifurcations with slight variations of $\xi.$ Analyzing the biological interpretations of these bifurcations is important as they cause dramatic changes in the system dynamics. After a certain $\xi$ level, multiple limit cycles appear around the second interior equilibrium point via a global hopf bifurcation. Herein, the limit cycle enclosing the second interior equilibrium point is essentially unstable and the outer one is locally stable. By smooth increment of $\xi$, the system  undergoes a homoclinic bifurcation where the limit cycles collide with both the stable and unstable manifolds that produce a single closed orbit that goes through the first interior equilibrium point. This is extremely advantageous for bio-control. As in the case of a homoclinic occurrence only initial data inside the homoclinic is attracted to the coexistence equilibrium that it encircles - the rest of the initial data in the phase is attracted to the pest-extinction state, see Fig. \ref{Bif4}. Future work could look to calculate the percentage of initial data of the phase (up to say carrying capacity of prey, and some reasonable bound for the predator) inside the homoclinic, as a function of the other system parameters. When our bifurcation parameter $\xi$ crosses a certain critical value, the system changes the number of equilibrium states from two, to one, then one to zero through the saddle-node bifurcation. For higher levels of AF, there exists no stable coexistence states and thereby the AF facilitates complete pest extinction.

In \cite{S02} it is claimed that in the case of one interior equilibrium, local stability of the interior implies global stability. This is clearly not true if two interiors exist, as seen via \cite{S23}. The local stability of a (solitary) interior equilibrium does not imply global stability in general, as proved in the case of both Gauss class and Holling-Tanner type predator-prey models \cite{HS89}, \cite{S23}. We conjecture this is true for the current system as well. A possible approach thus to disproving Theorem from \cite{S02}, is to perform a change of coordinates, and then to show that the (solitary) interior equilibrium could be surrounded by two limit cycles. Also, we note that via our approach, similar to the approach that the authors in \cite{S02} use to derive global stability of the pest free state (see Fig. \ref{figure3_1}), one can also have global stability of the pest free state (see Fig. \ref{figuretwoInt2}), for $\xi$ above a certain threshold.

The current work does not include Allee effects. This however is a well known extinction mechanism in population dynamics. Note, both the strong and weak Allee effects in the pest \cite{S43}
as well as a component Allee effect in the introduced predator \cite{S43}, have been considered in the context of additional food models. The weak Allee effect in the pest will enable complete extinction - that is $(0,0)$ can be stabilized, for large enough pest death rate, but $(0,y^{*})$ does not even exist in the case of a strong Allee effect \cite{S43}. In the case of a (component) Allee effect in the predator, neither complete extinction $(0,0)$ nor pest extinction can be stabilized, for any quantity of additional food. So combining Allee effect with additional food models actually hinder bio-control instead of enhancing it. An advantage of the current model is that it enables bi-stability dynamics, and so the arbitrarily large (pest) initial conditions, can be attracted to the pest extinction state, see Fig. \ref{Bif4}.
As future work one could consider the cumulative effects of AF, an Allee effect and mutual interference.

It is clear that the combination of mutual interference effect and the carefully chosen additional food level is the key to pest eradication. This scenario would certainly be profitable from an eco-management point of view to design eco-friendly prey control programs. However, it is crucial to design lab experiments by keeping the same agents (pest/predator and additional food), conditions and dynamics that is equivalent to the mathematical model derived in order to evaluate the biological applicability of the model. Another interesting direction would be to investigate the possible non standard bifurcations such as saddle-node-transcritical bifurcation in co-dimension two, pitchfork-transcritical bifurcation, in co-dimension two, cusp-transcritical bifurcation, in co-dimension two, \cite{S27}, \cite{S28}, \cite{S29}.


\section{Appendix}

\subsection{Proof of Theorem \ref{thmchap3_2} using the Hopf bifurcation theorem  \ref{thmchap3_1}}\label{ProveHoph}
\begin{proof}
We consider the parameter $\xi$ as the bifurcation parameter. Accordingly, Trace$\left(T\left(\xi\right)\right)=J_{11}+J_{22}$,\hspace{4mm} and Determinant $\left(D(\xi)\right)=J_{11}J_{22}-J_{12}J_{21}$  of the jacobian matrix are smooth functions of $\xi$.\\
The roots of the characteristic equation,  
\begin{equation}\label{proofthmchap3_2_1}
    \lambda^2-T\left(\xi\right)\lambda+D\left(\xi\right)=0
\end{equation} 
are,
$\lambda_1=p\left(\xi\right)+iq\left(\xi\right)$ and $\lambda_2=p\left(\xi\right)-iq\left(\xi\right)$ where $p\left(\xi\right)$ and $q\left(\xi\right)$ be smooth functions of the parameter $\xi$.\\
 By substituting $\lambda_1=p\left(\xi\right)+iq\left(\xi\right)$ into the equation~\eqref{proofthmchap3_2_2} we yield,
\begin{equation}\label{proofthmchap3_2_3}
   \left[p\left(\xi\right)+iq\left(\xi\right)\right]^2-T\left(\xi\right)\left[p\left(\xi\right)+iq\left(\xi\right)\right]+D\left(\xi\right)=0 
\end{equation}
Next, we take the derivative of the function in equation~\eqref{proofthmchap3_2_3} with respect to the control parameter $\xi$.\\
 $2\left[p\left(\xi\right)+iq\left(\xi\right)\right].\left[\Dot{p\left(\xi\right)}+i\Dot{q\left(\xi\right)}\right]-T\left(\xi\right)\left[\Dot{p\left(\xi\right)}+i\Dot{q\left(\xi\right)}\right]-\Dot{T\left(\xi\right)}\left[p\left(\xi\right)+iq\left(\xi\right)\right]+\Dot{D\left(\xi\right)}=0$
\begin{equation}\label{proofthmchap3_2_4}
   2p\left(\xi\right)\Dot{p\left(\xi\right)}-2q\left(\xi\right)\Dot{q\left(\xi\right)}-T\left(\xi\right)\Dot{p\left(\xi\right)}-\Dot{T\left(\xi\right)}p\left(\xi\right)+\Dot{D\left(\xi\right)}+i\left[2\Dot{p\left(\xi\right)}q\left(\xi\right)+2\Dot{q\left(\xi\right)}p\left(\xi\right)-T\left(\xi\right)\Dot{q\left(\xi\right)}-\Dot{T\left(\xi\right)} q\left(\xi\right)
\right]\\=0
\end{equation}

We use the  transformations of $a\left(\xi\right)=2p\left(\xi\right)- T\left(\xi\right),\hspace{1mm} b\left(\xi\right)=2q\left(\xi\right),\hspace{1mm} c\left(\xi\right)=-\Dot{T\left(\xi\right)}p\left(\xi\right)+\Dot{D\left(\xi\right)}, \hspace{1mm} d\left(\xi\right)=-\Dot{T\left(\xi\right)}q\left(\xi\right)$ to reduce the terms in the equation ~\eqref{proofthmchap3_2_4}.

\begin{equation}\label{proofthmchap3_2_5}
\left[a\left(\xi\right)\Dot{p\left(\xi\right)}-b\left(\xi\right)\Dot{q\left(\xi\right)}+c\left(\xi\right)\right]+ i\left[b\left(\xi\right)\Dot{p\left(\xi\right)}+a\left(\xi\right)\Dot{q\left(\xi\right)}+d\left(\xi\right)  \right]=0 
\end{equation}
Then, we obtain two sub equations by setting real and imaginary components to zero.
\begin{equation}\label{proofthmchap3_2_6}
a\left(\xi\right)\Dot{p\left(\xi\right)}-b\left(\xi\right)\Dot{q\left(\xi\right)}+c\left(\xi\right)=0 
\end{equation}
\begin{equation}\label{proofthmchap3_2_7}
b\left(\xi\right)\Dot{p\left(\xi\right)}+a\left(\xi\right)\Dot{q\left(\xi\right)}+d\left(\xi\right)=0 
\end{equation}
We combine equations, ~\eqref{proofthmchap3_2_6} and ~\eqref{proofthmchap3_2_7} in order to solve for $\Dot{p\left(\xi\right)}.$
\begin{equation}\label{proofthmchap3_2_8}
\Dot{p\left(\xi\right)}=-\frac{c\left(\xi\right)a\left(\xi\right)+d\left(\xi\right)b\left(\xi\right)}{(a(\xi))^{2}+(b(\xi))^{2}}
\end{equation}
Suppose that $T\left(\xi\right)=0$, at the parameter value, $\xi=\xi^{*},$ then the characteristic equation reduces to 
\begin{equation}\label{proofthmchap3_2_2}
    \lambda^2+D\left(\xi\right)=0
\end{equation}
The roots of the equation~\eqref{proofthmchap3_2_2} are $\lambda_1=i\sqrt{D\left(\xi\right)}$ and $\lambda_2=-i\sqrt{D\left(\xi\right)}$. Therefore we have purely imaginary eigenvalues. Now, we focus on validating the transversality conditions,
$\frac{d Re \lambda_{1}(\xi)}{d\xi}|_{\xi=\xi^{*}}=\Dot{p\left(\xi^{*}\right)}\neq 0$,\hspace{4mm}
$\frac{d Re \lambda_{2}(\xi)}{d\xi}|_{\xi=\xi^{*}}=\Dot{p\left(\xi^{*}\right)}\neq 0$.

To verify the occurrence of Hopf-bifurcation, we need to prove that\\
$\frac{d Re\lambda_{i}\left(\xi\right)}{d\xi}|_{\xi=\xi^{*}}=\Dot{p\left(\xi^{*}\right)}=-\frac{c\left(\xi^{*}\right)a\left(\xi^{*}\right)+d\left(\xi^{*}\right)b\left(\xi^{*}\right)}{a\left(\xi^{*}\right)^{2}+b\left(\xi^{*}\right)^{2}}\neq 0 $\\
That is, both numerator and denominator of $\Dot{p\left(\xi^{*}\right)}$ have to be non zero. 
\begin{equation}\label{proofthmchap3_2_9}
c\left(\xi^{*}\right)a\left(\xi^{*}\right)+d\left(\xi^{*}\right)b\left(\xi^{*}\right)\neq 0
\end{equation}
\begin{equation}\label{proofthmchap3_2_10}
a\left(\xi^{*}\right)^{2}+b\left(\xi^{*}\right)^{2}\neq 0 
\end{equation}
Now, we use the transformations introduced in ~\eqref{proofthmchap3_2_1} to simplify the left hand sides of equations ~\eqref{proofthmchap3_2_9} and ~\eqref{proofthmchap3_2_10}. Our expectation is to  show that the parametric constraints in theorem \ref{thmchap3_2}, verify the desired results. First, we  will obtain the parametric constraints for \eqref{proofthmchap3_2_9}\\

$c\left(\xi^{*}\right)a\left(\xi^{*}\right)+d\left(\xi^{*}\right)b\left(\xi^{*}\right)\\
=\left[2p\left(\xi\right)-T\left(\xi\right)\right]\left[-\Dot{T\left(\xi\right)}p\left(\xi\right)+\Dot{D\left(\xi\right)}\right]+ 2q\left(\xi\right)\left[-\Dot{T\left(\xi\right)}.q\left(\xi\right)\right]\\ 
=-2 p^{2}\left(\xi\right)\Dot{T\left(\xi\right)}+2D\left(\xi\right)\Dot{D\left(\xi\right)}+T\left(\xi\right)p\left(\xi\right)\Dot{T\left(\xi\right)}-T\left(\xi\right)\Dot{D\left(\xi\right)}-2 q^{2}\left(\xi\right)\Dot{T\left(\xi\right)}\\
=-2\left[p^{2}\left(\xi\right)+q^{2}\left(\xi\right)\right]\Dot{T\left(\xi\right)}+ 2p\left(\xi\right)\Dot{D\left(\xi\right)}+p\left(\xi\right)T\left(\xi\right)\Dot{T\left(\xi\right)}- T\left(\xi\right)\Dot{D\left(\xi\right)}$\\
We use the root of the equation of $\lambda$ in \eqref{proofthmchap3_2_1} 
 $\lambda_i=\frac{T\left(\xi\right)}{2}\pm \frac{\sqrt{(T\left(\xi\right))^{2}-4D\left(\xi\right)}}{2}=p\left(\xi\right)\pm iq\left(\xi\right)$ for $i=1,2.$ By comparing the real and the imaginary terms, we yield the substitutions $p\left(\xi\right)=\frac{T\left(\xi\right)}{2},\hspace{2mm}q\left(\xi\right)=\frac{\sqrt{(T\left(\xi\right))^{2}-4D\left(\xi\right)}}{2}$. We use these substitutions to simplify the expression as follows. \\
$=-2\left[\left(\frac{T\left(\xi\right)}{2}\right)^{2}+\left(\frac{\sqrt{(T\left(\xi\right))^{2}-4D\left(\xi\right)}}{2}\right)^{2}\right]\Dot{T\left(\xi\right)}+ 2\frac{T\left(\xi\right)}{2}\Dot{D\left(\xi\right)}+\frac{T\left(\xi\right)}{2}T\left(\xi\right)\Dot{T\left(\xi\right)}-T\left(\xi\right)\Dot{D\left(\xi\right)}\\
=-2\left[\frac{T^2\left(\xi\right)}{2} -D\right]+\frac{T^2\left(\xi\right)}{2}\Dot{T\left(\xi\right)}\\
= \left(2D- \frac{T^2\left(\xi\right)}{2}\right)\Dot{T\left(\xi\right)}$\\
Now we substitute the Trace and Determinants,  $T\left(\xi\right)=J_{11}+J_{22}$ and $D\left(\xi\right)=\left(J_{11}J_{22}-J_{12}J_{21}\right)\\
=\left [2\left(J_{11}J_{22}-J_{12}J_{21}\right)-\frac{1}{2} \left(J_{11}+J_{22}\right)^{2} \right]\Dot{T\left(\xi\right)}\\
=\left [2\left(J_{11}J_{22}-J_{12}J_{21}\right)-\frac{1}{2} \left(J_{11}^{2}+2J_{11}J_{22}+J_{22}^{2}\right) \right]\Dot{T\left(\xi\right)}\\
=\left [-2J_{12}J_{21}-\frac{1}{2} \left(J_{11}^{2}+J_{22}^{2}\right)+J_{11}J_{22}  \right]\Dot{T\left(\xi\right)}\\
=\left [-2J_{12}J_{21}-\frac{1}{2} \left(J_{11}-J_{22}\right)^2  \right]\Dot{T\left(\xi\right)}$
\begin{equation}\label{proofthmchap3_2_11}
    c\left(\xi^{*}\right)a\left(\xi^{*}\right)+d\left(\xi^{*}\right)b\left(\xi^{*}\right)=-\frac{1}{2} \left [\left(J_{11}-J_{22}\right)^2+ 4J_{12}J_{21} \right]\Dot{T\left(\xi\right)}
\end{equation}

Next we simplify the equation \eqref{proofthmchap3_2_10} by substituting for \\
$a\left(\xi\right)=2p\left(\xi\right)- T\left(\xi\right)=2\frac{T\left(\xi\right)}{2}-\Dot{T\left(\xi\right)}$ and $b\left(\xi^{*}\right)=2q\left(\xi\right)=2\frac{\sqrt{T\left(\xi\right)^{2}-4D\left(\xi\right)}}{2}$ from the transformations initialized.\\

$a\left(\xi^{*}\right)^{2}+b\left(\xi^{*}\right)^{2} = \left[2\frac{T\left(\xi\right)}{2}-T\left(\xi\right)\right]^2+\left[\frac{\sqrt{T\left(\xi\right)^{2}-4D\left(\xi\right)}}{2}\right]^{2}=T\left(\xi\right)^{2}-4D\left(\xi\right)\\
a\left(\xi^{*}\right)^{2}+b\left(\xi^{*}\right)^{2} =\left(J_{11}+J_{22}\right)^2-4\left(J_{11}J_{22}-J_{12}J_{21} \right) =\left(J_{11}-J_{22}\right)^2+4J_{12}J_{21}$
\begin{equation}\label{proofthmchap3_2_12}
   a\left(\xi^{*}\right)^{2}+b\left(\xi^{*}\right)^{2} = \left(J_{11}-J_{22}\right)^2+4J_{12}J_{21}
\end{equation}

By \eqref{proofthmchap3_2_9}, \eqref{proofthmchap3_2_10}, \eqref{proofthmchap3_2_11} and \eqref{proofthmchap3_2_12}, Hope bifurcation is possible if the system parameter satisfies the following conditions.
\begin{enumerate}
    \item $\left(J_{11}-J_{22}\right)^2+4J_{12}J_{21}\neq 0$
    \item $\Dot{T\left(\xi\right)} \neq 0$ 
\end{enumerate}

We first simplify the expression 
   $\left(J_{11}-J_{22}\right)^2+4J_{12}J_{21}\\
    = \left[1-\frac{2x_2^{*}}{k}-\frac{\left(1+\alpha\xi+\epsilon y_2^{*}\right)y_2^{*}}{\left(1+\alpha\xi+x_2^{*}+\epsilon y_2^{*}\right)^2}-\frac{\beta\left(x_2^{*}+\xi\right)\left(1+\alpha\xi+x_2^{*}\right) }{\left(1+\alpha\xi+x_2^{*}+\epsilon y_2^{*}\right)^2}+\delta\right]^2-4\left[\frac{\left(1+\alpha\xi+  x_2^{*}\right)x_2^{*}}{\left(1+\alpha\xi+x_2^{*}+\epsilon y_2^{*}\right)^2}\right].\left[\frac{\left(1+\alpha\xi+ \epsilon y_2^{*}-\xi\right)\beta y_2^{*}}{\left(1+\alpha\xi+x_2^{*}+\epsilon y_2^{*}\right)^2}\right]$\\

Since $\left(J_{11}-J_{22}\right)^2\geq 0$ we derive the necessary condition to satisfy (1), that is $4J_{12}J_{21}>0$\\
Under the constrain $\xi<1+\alpha \xi+y_2^*$, 
we deduce ,\\

$J_{12}.J_{21}= -\frac{\left(1+\alpha\xi+ x^*\right)x^*}{\left(1+\alpha\xi+x^*+\epsilon y^*\right)^2}.\frac{\left(1+\alpha\xi+ \epsilon y^*-\xi\right)\beta y^*}{\left(1+\alpha\xi+x^*+\epsilon y^*\right)^2}<0$ \\ 
 Thus, $\left(J_{11}-J_{22}\right)^2-4J_{12}J_{21}\neq 0$. \\

Next, we derive $\Dot{T\left(\xi\right)}$ and simplify the expression to prove the statement (2).\\

$\Dot{T\left(\xi\right)}=\frac{\partial \left(J_{11}+J_{22}\right)}{\partial \xi} =\frac{\beta\left[\left(1+\alpha\xi+x_2^*+\epsilon y_2^*\right).\left(\alpha x_2^*+x_2^*+2\alpha\xi+1\right)\right]-2\alpha\beta\left(x_2^*+\xi\right).\left(1+\alpha \xi +x_2^*\right)+\alpha y_2^* \left(1+\alpha\xi+\epsilon y_2^*-x_2^*\right) }{\left(1+\alpha\xi+x_2^*+\epsilon y_2^*\right)^3}$\\\\

Provided that, $y_2^*\neq\frac{\beta(x_2^*+\xi)}{\alpha}\left[\frac{2\alpha\delta(1+\alpha+x_2^*)-\beta(\alpha x_2^*+x_2^*+2\alpha+1)}{(\beta-2\delta)x_2^*+\beta\xi}\right],$ \\
we finally obtain $\Dot{T\left(\xi\right)} \neq 0$, This completes the proof.
\end{proof}

\subsection{Proof of Theorem \ref{thmchap3_11} using the Sotomayor's Theorem \ref{thmchap3_10}}\label{saddleNodeProof}
\begin{proof}

\begin{enumerate}
    \item We first prove the condition 
 \begin{equation}\label{thmchapPro3_4_2}
     w^{T}g_{\xi}\left(\Tilde{E},\xi^{SN}\right)\neq0
 \end{equation} 
Transpose of the Jacobian Matrix $J^{T}=
 \begin{bmatrix} J_{11} & J_{21} \\ J_{12} & J_{22} \end{bmatrix} $
 The eigen values of $J^{T}$ are: $\lambda_i^{T}=\frac{1}{2}\left[J_{11}+J_{22}\mp \sqrt{J_{11}^2+4J_{12}J_{21}-2J_{11}J_{22}+{J_{22}}^2}\right]$ and the corresponding eigen vectors are:
$w_i=
 \begin{bmatrix} \frac{J_{11}-J_{22}\mp\sqrt{J_{11}^2+4J_{12}J_{21}-2J_{11}J_{22}+{J_{22}}^2}}{2J_{12}}\\ 1 \end{bmatrix},\hspace{3mm}$ for $i=1,2.$ \\
 To get a simple zero eigen value we make $\lambda_i^{T}=0,$ that is 
$J_{11}+J_{22}=0$ that is $ J_{22}=-J_{11}$ and
$\sqrt{J_{11}^2+4J_{12}J_{21}-2J_{11}J_{22}+{J_{22}}^2}=\sqrt{(J_{11}-J_{22})^2+4J_{12}J_{21}}=0 $ that implies ${J_{11}}^2=-J_{12}J_{21}$\\
Now using the above simplified expressions, the eigen vector corresponding to the simple zero eigen value is,
\begin{equation} \label{wt}
   w^T=  \begin{bmatrix} w_1 \\ w_2 \end{bmatrix}=\begin{bmatrix} \frac{-J_{22}}{J_{12}} \\ 1 \end{bmatrix}
 \end{equation}
Now we suppose 
 $g=
 \begin{bmatrix} \frac{dx}{dt}\\  \frac{dy}{dt} \end{bmatrix}_{(\Tilde{x},\Tilde{y})}= 
\begin{bmatrix} 
\Tilde{x}(1-\frac{\Tilde{x}}{k}) - \frac{\Tilde{x}\Tilde{y}}{1+\alpha \xi+\Tilde{x}+\epsilon \Tilde{y}} \\
\frac{\beta (x+\xi)\Tilde{y}}{1+\alpha \xi+\Tilde{x}+\epsilon \Tilde{y}}-\delta \Tilde{y}
\end{bmatrix}$
\\
and, the partial derivative of $g$ is, 
\begin{equation} \label{gxi}
    g_{\xi}=\frac{\partial g}{\partial \xi}=\begin{bmatrix} 
\frac{\alpha \Tilde{x}\Tilde{y}}{\left(1+\alpha \xi+\Tilde{x}+\epsilon \Tilde{y}\right)^2} \\ \frac{\beta \Tilde{y}}{1+\alpha \xi+\Tilde{x}+\epsilon \Tilde{y}}-\frac{\alpha\beta(\Tilde{x}+\xi)\Tilde{y}}{\left(1+\alpha \xi+\Tilde{x}+\epsilon \Tilde{y}\right)^2}
\end{bmatrix}
\end{equation}

Now, we derive the expression for $w^{T}g_{\xi}\left(\Tilde{E},\xi^{SN}\right)$ using the equations \eqref{wt} and \eqref{gxi}
\begin{equation}\label{thmchapPro3_4_3}
      w^{T}g_{\xi}\left(\Tilde{E},\xi^{SN}\right)=\begin{bmatrix} \frac{-J_{22}}{J_{12}} \\ 1 \end{bmatrix}\begin{bmatrix} 
\frac{\alpha \Tilde{x}\Tilde{y}}{\left(1+\alpha \xi+\Tilde{x}+\epsilon \Tilde{y}\right)^2} \\ \frac{\beta \Tilde{y}}{1+\alpha \xi+\Tilde{x}+\epsilon \Tilde{y}}-\frac{\alpha\beta(\Tilde{x}+\xi)\Tilde{y}}{\left(1+\alpha \xi+\Tilde{x}+\epsilon \Tilde{y}\right)^2}\end{bmatrix}
\end{equation}
By applying the equations \eqref{proof3.1_3} and \eqref{proof3.1_3_1} we simplify the expression $w^{T}g_{\xi}\left(\Tilde{E},\xi^{SN}\right).$
  \begin{equation} \label{thmchapPro3_4_4}
      w^{T}g_{\xi}\left(\Tilde{E},\xi^{SN}\right)= \frac{\delta}{\beta(\Tilde{x}+\xi)}\left[\frac{\delta\left(1+\alpha \xi+\Tilde{x}\right)-\beta (\Tilde{x}+\xi)}{1+\alpha \xi+\Tilde{x}}  +(\beta-\alpha\delta)\Tilde{y}\right]
  \end{equation}

 WLOG $(\beta-\alpha\delta)\xi-\delta>0$  and $\beta>0$.   \\
 $w^{T}g_{\xi}\left(\Tilde{E},\xi^{SN}\right)\neq 0$ given that $\Tilde{y}\neq\frac{(\beta-\delta)\Tilde{x}+(\beta-\delta\alpha)\xi-\delta}{(1+\alpha \xi+\Tilde{x})(\beta-\delta\alpha)}$\\

\item Next, we show the second condition, \hspace{2mm} $w^{T}\left[D^{2}g\left(\Tilde{E},\xi^{SN}\right)\left(v,v\right)\right]\neq 0$ \\

Now we Consider $f^{(1)}(x,y)=\frac{dx}{dt},$ \hspace{1mm} $f^{(2)}(x,y)=\frac{dy}{dt},$ \hspace{1mm}$v=\begin{bmatrix} v_1 \\ v_2\end{bmatrix}=\begin{bmatrix} \frac{J_{11}}{J_{21}} \\ 1\end{bmatrix},$ and
$w=\begin{bmatrix} w_1 \\ w_2\end{bmatrix}=\begin{bmatrix} \frac{-J_{22}}{J_{12}} \\ 1\end{bmatrix}$ 
where $v $ and $w$ are the eigen vectors corresponding to the simple zero eigen values of the jacobian matrix $J$ and $J^T$ respectively. Thereby, we get the connections, $J_{11}=J_{22}$ and $J_{11}J_{22}=J_{12}J_{21}$ ,  \\
$w^{T}\left[D^{2}g\left(\Tilde{E},\xi^{SN}\right)\left(v,v\right)\right]=w^{T}\begin{bmatrix}.
\frac{\partial^2 f^{(1)}(\Tilde{x},\Tilde{y})}{\partial x^2}v_1^{2}+2\frac{\partial^2 f^{(1)}(\Tilde{x},\Tilde{y})}{\partial x \partial y}v_1.v_2+\frac{\partial^2 f^{(1)}(\Tilde{x},\Tilde{y})}{\partial y^2}v_2^{2}\\ \frac{\partial^2 f^{(2)}(\Tilde{x},\Tilde{y})}{\partial x^2}v_1^{2}+2\frac{\partial^2 f^{(2)}(\Tilde{x},\Tilde{y})}{\partial x \partial y}v_1.v_2+\frac{\partial^2 f^{(2)}(\Tilde{x},\Tilde{y})}{\partial y^2}v_2^{2}
\end{bmatrix}$\\
$w^{T}\left[D^{2}g\left(\Tilde{E},\xi^{SN}\right)\left(v,v\right)\right]= -\left(\frac{J_{22}}{J_{12}}\right)\left(\frac{J_{11}}{J_{21}}\right)^2a_1+2a_2\left(\frac{-J_{22}}{J_{12}}\right)\left(\frac{J_{11}}{J_{21}}\right)+ \frac{-J_{22}}{J_{12}}a_3+\left(\frac{J_{11}}{J_{21}}\right)^{2}
 a_4+2\frac{J_{11}}{J_{21}} a_5 +a_6$\\
 
where the higher order derivative terms are listed as below.\\
$ a_1=\frac{\partial^2f^{(1)}(\Tilde{x},y)}{\partial {x^2}}=-\frac{2}{k}-\frac{2\Tilde{x}\Tilde{y}}{\left(1+\alpha \xi+\Tilde{x}+\epsilon \Tilde{y}\right)^3}+\frac{2\Tilde{y}}{\left(1+\alpha \xi+\Tilde{x}+\epsilon \Tilde{y}\right)^2}\\
a_2=\frac{\partial^2f^{(1)}(\Tilde{x},\Tilde{y})}{\partial x\partial y}=-\frac{2\epsilon \Tilde{x} \Tilde{y}}{\left(1+\alpha \xi+\Tilde{x}+\epsilon \Tilde{y}\right)^3}+\frac{\Tilde{x}+\epsilon \Tilde{y}}{\left(1+\alpha \xi+\Tilde{x}+\epsilon \Tilde{y}\right)^2}-\frac{1}{\left(1+\alpha \xi+\Tilde{x}+\epsilon \Tilde{y}\right)}\\
a_3=\frac{\partial^2f^{(1)}(\Tilde{x},\Tilde{y})}{\partial y^2}=-\frac{2\epsilon^2 \Tilde{x} \Tilde{y}}{\left(1+\alpha \xi+\Tilde{x}+\epsilon \Tilde{y}\right)^2}+\frac{2\epsilon \Tilde{x}}{\left(1+\alpha \xi+x+\epsilon \Tilde{y}\right)}\\
a_4=\frac{\partial^2f^{(2)}(\Tilde{x},\Tilde{y})}{\partial x^2}\frac{2\beta(\Tilde{x}+\xi)\Tilde{y}}{\left(1+\alpha \xi+\Tilde{x},\Tilde{y}+\epsilon \Tilde{y}\right)^3}-\frac{2\beta \Tilde{y}}{\left(1+\alpha \xi+\Tilde{x}+\epsilon \Tilde{y}\right)^2}\\
 a_5=\frac{\partial^2f^{(2)}(\Tilde{x},\Tilde{y})}{\partial x\partial y}=\frac{2\beta\epsilon(x+\xi)\Tilde{y}}{\left(1+\alpha \xi+\Tilde{x}+\epsilon \Tilde{y}\right)^3}-\frac{(\Tilde{x}+\xi+\beta\epsilon) \Tilde{y}}{\left(1+\alpha \xi+\Tilde{x}+\epsilon \Tilde{y}\right)^2}-\frac{\beta}{\left(1+\alpha \xi+\Tilde{x}+\epsilon \Tilde{y}\right)}\\
a_6=\frac{\partial^2f^{(2)}(\Tilde{x},\Tilde{y})}{\partial y^2}=\frac{2\beta\epsilon^2(\Tilde{x}+\xi)\Tilde{y}}{\left(1+\alpha \xi+\Tilde{x}+\epsilon \Tilde{y}\right)^3}-\frac{2\beta\epsilon(\Tilde{x}+\xi)}{\left(1+\alpha \xi+\Tilde{x}+\epsilon \Tilde{y}\right)^2}$\\

Now, the simplified formula for the $ w^{T}\left[D^{2}g\left(\Tilde{E},\xi^{SN}\right)\left(v,v\right)\right]$  is obtained in equation ~\eqref{thmchapPro3_4_6}.
\begin{equation}\label{thmchapPro3_4_6}
    w^{T}\left[D^{2}g\left(\Tilde{E},\xi^{SN}\right)\left(v,v\right)\right]= - \frac{J_{11}}{J_{21}}a_1-2a_2+ \frac{J_{11}}{J_{12}}a_3+\left(\frac{J_{11}}{J_{21}}\right)^{2}
 a_4+2\frac{J_{11}}{J_{21}} a_5 +a_6
\end{equation}
We obtain the sufficient conditions for $w^{T}\left[D^{2}g\left(\Tilde{E},\xi^{SN}\right)\left(v,v\right)\right]\neq 0$ by breaking the expression \eqref{thmchapPro3_4_6} into the following inequalities.
\begin{enumerate}[label=\roman*]
    \item Showing that $-\frac{J_{11}}{J_{21}}a_1+\left(\frac{J_{11}}{J_{21}}\right)^{2}a_4>0.$
    \item Showing that $-2a_2+a_6>0.$
    \item Showing that $\frac{J_{11}}{J_{12}}a_3+2\frac{J_{11}}{J_{21}} a_5>0 $
\end{enumerate}
 We simplify the expressions by taking $M=1+\alpha \xi+\Tilde{x}+\epsilon \Tilde{y}$ and prove (i)-(iii)\\ 

 \textbf{(i) Proving the expression $\frac{-J_{11}}{J_{21}}a_1+\left(\frac{J_{11}}{J_{21}}\right)^{2}a_4>0$}\\
First we simplify $\frac{J_{11}}{J_{21}}$
\begin{equation}\label{thmchapPro3_4_7}
\frac{J_{11}}{J_{21}}=\frac{\left(k-2\Tilde{x}\right)M^2-\left(1+\alpha \xi+\epsilon \Tilde{y}\right)k \Tilde{y}}{\beta k\left(1+\alpha \xi+\epsilon \Tilde{y}-\xi\right)\Tilde{y}}
\end{equation}

Sufficient conditions for $\frac{J_{11}}{J_{21}}>0 $ is both the numerator and denominator being positive. 
First we simplify the denominator and numerator by substituting for terms from equations \eqref{proof3.1_3} and \eqref{proof3.1_3_1} 
\begin{equation}\label{j11overj21}
\frac{J_{11}}{J_{21}}=\frac{\beta^2(k-2\Tilde{x})(\Tilde{x}+\xi)^2-\delta k\Tilde{y}\left[(\beta-\delta)\Tilde{x}+\beta\xi\right]}{\left[(\beta-\delta)\Tilde{x}+\beta\xi\right]\beta\delta k}
\end{equation}

WLOG $\beta-\delta>0,$ therefore the denominator is positive. The numerator also positive if

\begin{equation}\label{condition2}
\Tilde{y}<\frac{\beta^2(k-2\Tilde{x})(\Tilde{x}+\xi)^2}{\delta k\left[(\beta-\delta)\Tilde{x}+\beta\xi\right]}
\end{equation}
Next we derive the condition for $\frac{-J_{11}}{J_{21}}a_1+\left(\frac{J_{11}}{J_{21}}\right)^{2}a_4>0$\\

 $\frac{-J_{11}}{J_{21}}a_1+\left(\frac{J_{11}}{J_{21}}\right)^{2}a_4
 =\frac{J_{11}}{J_{21}}\left[\frac{2}{k}+\frac{2\Tilde{x}\Tilde{y}}{M^3}-\frac{2\Tilde{y}}{M^2}+\left(\frac{\left(k-2\Tilde{x}\right)M^2-\left(1+\alpha\xi+\epsilon \Tilde{y}\right)k \Tilde{y}}{k\beta \Tilde{y}\left(1+\alpha\xi+\epsilon \Tilde{y}\right)}\right)
\left(\frac{2\beta(\Tilde{x}+\xi)\Tilde{y}}{M^3}-\frac{2\beta \Tilde{y}}{M^2}\right)\right]\\
=\frac{2J_{11}}{J_{21}}\left[\frac{1}{k}-\frac{\left(1+\alpha\xi+\epsilon \Tilde{y}\right)}{M^3}+\left[\frac{\left(k-2\Tilde{x}\right)M^2-\left(1+\alpha\xi+\epsilon \Tilde{y}\right)k \Tilde{y}}{k\beta \Tilde{y}\left(1+\alpha\xi+\epsilon \Tilde{y}\right)}\right]\\
\frac{\beta \Tilde{y}\left(\xi-\left(1+\alpha\xi+\epsilon \Tilde{y}\right)\right)}{M^3}\right]\\
=\frac{2J_{11}}{J_{21}}\left[\frac{1}{k}-\frac{\left(1+\alpha\xi+\epsilon \Tilde{y}\right)}{M^3}+\frac{\left(k-2\Tilde{x}\right)\left(\xi-(1+\alpha\xi+\epsilon\Tilde{y})\right)}{kM(1+\alpha\xi+\epsilon\Tilde{y})}+\frac{\left(1+\alpha\xi+\epsilon\Tilde{y})-\xi\right)\Tilde{y}}{M^3}\right]\\
=\frac{2J_{11}}{J_{21}}\left[\frac{M\left(1+\alpha\xi+\epsilon \Tilde{y}\right)-\left(k-2\Tilde{x}\right)\left(\xi-(1+\alpha\xi+\epsilon\Tilde{y})\right)}{kM(1+\alpha\xi+\epsilon\Tilde{y})}+\frac{\Tilde{y}((1+\alpha\xi+\epsilon\Tilde{y})-\xi)-(1+\alpha\xi+\epsilon\Tilde{y})}{M^3}\right]$\\
Now we substitute for $M=1+\Tilde{x}+\alpha\xi+\epsilon\Tilde{y}$ and $1+\alpha\xi+\epsilon\Tilde{y}$ from the equation \eqref{proof3.1_3} and \eqref{proof3.1_3_1}to obtain the simplified equation.
\begin{equation}\label{condition3}
 =\frac{2J_{11}}{J_{21}}\left[\frac{\beta(\Tilde{x}+\xi)\left((\beta-\delta)\Tilde{x}+\beta\xi\right)+\left(k-2\Tilde{x}\right)\left(\beta-\delta\right)(\Tilde{x}+\xi)}{\delta\left[(\beta-\delta)\Tilde{x}+\beta\xi\right]}+ \delta^2\frac{\Tilde{y}\left((\beta-\delta)\Tilde{x}+\beta\xi\right)-(\beta-\delta)(\Tilde{x}+\xi)}{\beta^3(\Tilde{x}+\xi)^3}
\right]  
\end{equation}
\end{enumerate}

Since $\beta-\delta$ from the equation \eqref{condition2} and \eqref{condition3}, we claim the necessary conditions for the positivity of the terms $\frac{-J_{11}}{J_{21}}a_1+\left(\frac{J_{11}}{J_{21}}\right)^{2}a_4$
as follows.
\begin{itemize}
    \item $\Tilde{y}<\frac{\beta^2(k-2\Tilde{x})(\Tilde{x}+\xi)^2}{\delta k\left[(\beta-\delta)\Tilde{x}+\beta\xi\right]}$
    \item $k>2\Tilde{x}$
    \item $\Tilde{y}>\frac{(\beta-\delta)\Tilde{x}+\beta\xi}{(\beta-\delta)(\Tilde{x}+\xi)}$
\end{itemize}

 \textbf{(ii)Proving the expression $-2a_2+a_6>0.$}\\
 
 $-2a_2+a_6=-2\left[\frac{-2\epsilon \Tilde{x} \Tilde{y}}{M^3}+\frac{\Tilde{x}}{M^2}+\frac{\epsilon y}{M^2}-\frac{1}{M}\right]+\frac{2\epsilon^2\beta(\Tilde{x}+\xi)}{M^3}-\frac{2\epsilon\beta(\Tilde{x}+\xi)}{M^2}$\\
 $=2\epsilon\Tilde{y}\left[\frac{2\Tilde{x}+\beta\epsilon(\Tilde{x}+\xi)}{M^3}\right]+\frac{2}{M^2}
\left[\frac{M-(\Tilde{x}+\epsilon\Tilde{y})-\beta\epsilon(\Tilde{x}+\xi)}{M^2}\right]\\
 -2a_2+a_6>0$ under the sufficient condition $M-(\Tilde{x}+\epsilon\Tilde{y})-\beta\epsilon(\Tilde{x}+\xi)>0,$
That is $\Tilde{y}<\frac{\beta\epsilon(\Tilde{x}+\xi)(1-\delta\epsilon)-\delta\Tilde{x}}{\delta\epsilon}$\\

 \textbf{(iii)Proving the expression $\frac{J_{11}}{J_{12}}a_3+2\frac{J_{11}}{J_{21}} a_5>0 $}\\
 
$\frac{J_{11}}{J_{12}}a_3+2\frac{J_{11}}{J_{21}} a_5\\
=-\left[\frac{\delta M^2-\beta(\Tilde{x}+\xi)M+\beta\epsilon(\Tilde{x}+\xi)\Tilde{y}}{\Tilde{x}\left(1+\alpha\xi+\Tilde{x}\right)}\right].
\underbrace{\left[\frac{-2\epsilon ^2 \Tilde{x}\Tilde{y}}{M^2}+\frac{2\epsilon \Tilde{x}}{M}\right]}_{\frac{2\epsilon}{M^2}\left[-\Tilde{x}\Tilde{y}+\left(1+\alpha\xi+\Tilde{x}+\epsilon \Tilde{y}\right)\Tilde{x}\right]\frac{2\epsilon \Tilde{x}}{M^2}\left(1+\alpha \xi\right)}+2\left[\frac{2\beta\epsilon(\Tilde{x}+\xi)\Tilde{y}}{M^3}-\frac{\beta\epsilon \Tilde{y}}{M^2}-\frac{\beta(\Tilde{x}+\xi)\Tilde{y}}{M^2}+\frac{\beta}{M}\right]
\\
=-\left[-\frac{\delta M^2+\beta(\Tilde{x}+\xi)M-\beta\epsilon(\Tilde{x}+\xi)\Tilde{y}}{\Tilde{x}\left(1+\alpha\xi+\Tilde{x}\right)}\right].\frac{2\epsilon \Tilde{x}}{M^2}\left(1+\alpha\xi+\Tilde{x}\right)+2\left[\frac{2\beta\epsilon(\Tilde{x}+\xi)\Tilde{y}}{M^3}-\frac{\beta\epsilon \Tilde{y}}{M^2}-\frac{\beta(\Tilde{x}+\xi)\Tilde{y}}{M^2}+\frac{\beta}{M}\right]\\
=2\left[ \underbrace{-\delta \epsilon}_{(1)}+\underbrace{\frac{\beta\epsilon(\Tilde{x}+\xi)}{M}}_{(2)}-\underbrace{\frac{\beta\epsilon^2(\Tilde{x}+\xi)\Tilde{y}}{M^2}}_{(3)}+\underbrace{\frac{2\beta\epsilon(\Tilde{x}+\xi)\Tilde{y}}{M^3}}_{(4)}-\underbrace{\frac{\beta\epsilon \Tilde{y}}{M^2}}_{(5)}-\underbrace{\frac{\beta(\Tilde{x}+\xi)}{M^2}}_{(6)}+\underbrace{\frac{\beta}{M}}_{(7)}\right] $\\

Let's simplify the expression,$(1)^{'}+(2)^{'}$\\
$-\delta \epsilon +\frac{\beta\epsilon(\Tilde{x}+\xi)}{M^2}=\frac{\epsilon}{M}\left[-\delta\left(1+\alpha\xi+\Tilde{x}+\epsilon \Tilde{y}\right)+\beta(\Tilde{x}+\xi)\right]\\$
By \eqref{proof3.1_3}, $(1)^{'}+(2)^{'}=0$\\

Let's simplify $(3)^{'}+(4)^{'}+(5)^{'}+(6)^{'}+(7)^{'}\\
=\frac{\beta\epsilon^2(\Tilde{x}+\xi)\Tilde{y}}{M^2}-\frac{\beta\epsilon \Tilde{y}}{M^2}-\frac{\beta(\Tilde{x}+\xi)}{M^2}+\frac{\beta}{M}+\frac{2\beta\epsilon(\Tilde{x}+\xi)\Tilde{y}}{M^3}\\
=\frac{\beta}{M^2}\left[-\epsilon^2(\Tilde{x}+\xi)\Tilde{y}-\epsilon \Tilde{y}- (\Tilde{x}+\xi)+1+\alpha\xi+\Tilde{x}+\epsilon \Tilde{y}\right]+\frac{2\beta\epsilon(\Tilde{x}+\xi)\Tilde{y}}{M^3}\\
=\frac{\beta}{M^3}\left[-\epsilon^2(\Tilde{x}+\xi)\Tilde{y}-\xi+1+\alpha\xi+2\epsilon \Tilde{x}\Tilde{y}+2\epsilon\xi \Tilde{y}\right]\\
=\frac{\beta}{M^3}\left[\epsilon \Tilde{x}\Tilde{y} \left(2-\epsilon\right)+\xi\left(\epsilon \Tilde{y}\left(2-\epsilon\right)-1+\alpha\right)+1\right]\\
=\frac{\beta}{M^3}\left[\underbrace{\epsilon \Tilde{x} \Tilde{y}\left(2-\epsilon\right)}_{(+) if \epsilon<1}+\underbrace{\xi\left(\epsilon \Tilde{y}\left(2-\epsilon\right)-1+\alpha\right)+1}_{(+) if \Tilde{y}>\frac{\left(1+\alpha\right)\xi-1}{\left(2-\epsilon\right)\epsilon}}\right]$\\
Hence,  $\frac{J_{11}}{J_{12}}a_3+2\frac{J_{11}}{J_{21}} a_5>0 $ under the conditions of $\Tilde{y}>\frac{\left(1+\alpha\right)\xi-1}{\left(2-\epsilon\right)\epsilon}$ \\
Hence, the proof is concluded.
\end{proof}

\subsection{Slope comparison}

Geometrically, the non-trivial prey nullcline,$y= \frac{\left(k-x\right)\left( 1+\alpha \xi+x \right)}{k-\left(k-x\right)\epsilon}$ has a hyperbolic shape that goes through the positive $y-$axis at the point $\left(0,\frac{1+\alpha\xi}{1-\epsilon}\right)$ and the positive $x-$ axis at the point $\left(k,0\right).$

The slope of the prey nullcline at the point $\left(0,\frac{1+\alpha\xi}{1-\epsilon} \right)$ can be computed by taking the derivative of the prey nullcline and then setting $x$ of that expression to zero.

$ \frac{dy(x)}{dx}=\frac{k(k-x)-k(1+\alpha\xi+x)-(k-x)^{2}\epsilon}{\left[k-(k-x)\epsilon\right]^2}.$\\
The slope at the point $\left(0,\frac{1+\alpha\xi}{1-\epsilon} \right)$ is $ m=\frac{k(1-\epsilon)-(1+\alpha\xi)}{k\left(1-\epsilon\right)^2}$ \\

There exists two interior equilibrium points when the slope $m$ exceeds the slope of the predator nullcline \eqref{eqn:chap3_19}. That is, \\
\begin{equation} \label{eqslopecom1}
    m_1 > m \Longleftrightarrow
\frac{k(1-\epsilon)-(1+\alpha\xi)}{k\left(1-\epsilon\right)^2} > \frac{\beta-\delta}{\delta \epsilon}
\end{equation}
,this yield
\begin{equation} \label{eqslopecom2}
 \xi<\frac{k(1-\epsilon)\left[\delta-\beta(1-\epsilon)\right]-\delta \epsilon}{\alpha \delta \epsilon}
\end{equation}


\end{document}